\documentclass[english,reqno, 11pt]{amsart}
\usepackage{latexsym}
\usepackage{amssymb,amsbsy,amsmath,amsfonts,amssymb,amscd}
\usepackage{epsfig, graphicx,color}
\usepackage{titletoc}
\usepackage{epsf}

\usepackage{subfigure}

\graphicspath{{../Figures/}}

\newtheorem{lem}{Lemma}[section]
\newtheorem{thm}{Theorem}[section]
\newtheorem{prop}[thm]{Proposition}

\numberwithin{equation}{section}

\newcommand{\del}{\partial}

\newcommand{\commentout}[1]{}
\newcommand{\vpran}[1]{\left(#1\right)}
\newcommand{\R}{\mathbb{R}}

\newcommand{\VV}{\mathbb{V}}
\newcommand{\dm}{ \, {\rm d} m}
\newcommand{\X}{{\mathbb X}}
\newcommand{\dmu}{\, {\rm d} \mu}
\newcommand{\dnu}{\, {\rm d} \nu}
\newcommand{\ds}{\, {\rm d} s}
\newcommand{\norm}[1]{\left\lVert#1 \, \right\rVert}
\newcommand{\nn}{\nonumber}

\newcommand {\Chi} {{\bf \raise 2pt \hbox{$\chi$}} }
\newcommand{\CalL}{{\mathcal{L}}}
\newcommand{\CalP}{{\mathcal{P}}}
\newcommand{\CalT}{{\mathcal{T}}}
\newtheorem{rmk}{Remark}[section]
\newcommand{\Ni}{\noindent}
\newcommand{\dsigma}{\, {\rm d}\sigma}

\newcommand {\dt}{\, {\rm d} t}

\newcommand {\dv}  { {\rm d}v} 
\newcommand {\da}  { {\rm d} a} 
\newcommand{\dx}{ \, {\rm d} x}
\newcommand{\dtau}{\, {\rm d} \tau}
\newcommand{\BigO}{{\mathcal{O}}}
\newcommand{\vint}[1]{\left\langle#1\right\rangle}
\newcommand{\Disp}{\displaystyle}

\newcommand {\f}   {\frac}
\newcommand {\p}   {\partial}
\newcommand{\barint}{\kern3pt \raise3.4pt\hbox{\vrule height.6pt
    width7pt} \kern-10pt \int}
    \newcommand{\Eps}{\epsilon}
    \newcommand{\Denote}{\stackrel{\Delta}{=}}
\newcommand{\abs}[1]{\lvert#1\rvert}
\newcommand{\beq}{\begin{equation}}
\newcommand{\eeq}{\end{equation}}
\newcommand{\bea} {\begin{array}{rl}}
\newcommand{\eea} {\end{array}}
\newcommand{\bepa}{\left\{ \begin{array}{l}}
\newcommand{\eepa} {\end{array}\right.}

\newtheorem{theorem}{Theorem}[section]

\newtheorem{definition}[theorem]{Definition}
\newtheorem{remark}[theorem]{Remark}
\newtheorem{proposition}[theorem]{Proposition}

\newcommand{\QED}{{ \hfill
                       {\unskip\kern 6pt\penalty 500 \raise -2pt\hbox{\vrule\vbox to 6pt{\hrule width 6pt
                       \vfill\hrule}\vrule} \par}   }}
\title[Large Gradient]{Macroscopic Limits of pathway-based kinetic models for E.coli chemotaxis in large gradient environments}

\author{Weiran Sun}
\address{Department of Mathematics, Simon Fraser University, 8888 University Dr., Burnaby, BC V5A 1S6, Canada}
\email{weirans@sfu.ca}

\author{Min Tang}
\address{Department of mathematics and Institute of natural sciences , 
Shanghai Jiao Tong University, Shanghai, 200240, China.}
\email{tangmin@sjtu.edu.cn}

\begin{document}

\maketitle


\begin{abstract}
It is of great biological interest to understand the molecular origins of chemotactic behavior of E. coli by developing population-level models based on the underlying signaling pathway dynamics. 
 We derive macroscopic models for E.coli chemotaxis that match quantitatively with the agent-based model (SPECS) for all ranges of the spacial gradient, in particular when the chemical gradient is large such that the standard Keller-Segel model is no longer valid. These equations are derived both formally and rigorously as asymptotic limits for pathway-based kinetic equations. We also present numerical results that show good agreement between the macroscopic models and SPECS.  Our work provides an answer to the question of how to determine the population-level diffusion coefficient and drift velocity from the molecular mechanisms of chemotaxis,
for both shallow gradients and large gradients environments.

\end{abstract}

\bigskip

\\[3mm]

\section{Introduction}
The movement of Escherichia coli (E. coli) presents an pattern of alternating forward-moving runs and reorienting tumbles. 
The run-and-tumble movements can be described by
a Boltzmann type velocity jump model \cite{HO,HP09}. 
It responds to external chemical signals by a biased random walk process.  In order to develop quantitative and predictive models, we have to first understand the response of bacteria to signal changes which is a sophisticated chemotactic signal transduction pathway. Fortunately, modern experimental technologies have enabled people to quantitatively measure the details of the E.coli chemotactic sensory system \cite{CSL,hazel,SB,SWOT}.  The response of E.coli to signal changes includes two steps: excitation and adaptation. Excitation is a rapid response of the cell to the external signal. It is due to the biochemical pathways regulating the flagellar motors. The slow adaptation allows the cell to subtract out the background signal.
It is carried out by the relatively slow receptor methylation and demethylation processes that modulate the methylation level of receptors \cite{E, tu,STB}.

It is possible to develop predictive agent-based models thanks to the understanding of the intracellular signalling pathway. However, direct computation of agent-based models is extremely time consuming when large number of cells are evolved. Moreover, the results are usually shown to be noisy  \cite{JOT}. Therefore, it is of great biological interest to understand the molecular origins of chemotactic behaviour of E. coli by developing population-level model based on the underlying signalling pathway dynamics.

In order to establish quantitative connections between the agent-based models and population-level models, the usual strategy is to use mesoscopic kinetic-transport equations and derive their macroscopic limits \cite{TMPA}.
There are two different classes of kinetic-transport
models for E.coli chemotaxis in the literature. One heuristically includes tumbling frequencies depending on the path-wise gradient of chemotactic signals, while the other takes into account an intra-cellular molecular biochemical pathway and relates the tumbling frequency to this information \cite{PTV}. It is possible to rescale both type of kinetic-transport models and study their diffusion and hyperbolic limits as in \cite{CMPS,DS,HO,HO1,STY,LTY}.

In most previous work, in order to take into account the effects of the internal signal pathway and derive macroscopic models, moments of the 
internal state are used \cite{ErbanOther04,STY, SWOT,XO,X,XY}.
The derivation is based on moment closure techniques. 
The moment system is usually closed by the assumption that the deviation of the internal state is not far from its expectation. This assumption is only valid when the chemical gradient is small. Macroscopic models are then derived by various asymptotic limits of the closed moment system. For example, the Keller-Segel model can be considered as a diffusion limit of the first-order moment closure, which assumes that the internal states of all bacteria are concentrated at their expectation. In \cite{ErbanOther04,XO} the authors derived the macroscopic Keller-Segel equation from agent-based models by incorporating a toy linear model for the intracellular signal transduction pathways. Recently, more complicated real intracellular signalling networks that are intrinsically nonlinear are considered \cite{STY,X}. 
Another model is introduced in \cite{SWOT}, where the authors developed a pathway-based mean field theory (PBMFT). This theory is used to explain the counter-intuitive experiment which shows the mass centre of the cells does not follow the dynamics of ligand concentration in a spatial-temporal fast-varying environment \cite{ZSOT}. PBMFT can be considered as a hyperbolic limit of the second order moment closure system \cite{STY}. However, compared with the agent-based simulations, PBMFT only recovers the right behavior for the mass centre in the spatial-temporal fast-varying environment and does not give a good match of the detailed dynamics of bacteria space distribution.  One remedy proposed in \cite{XY} is to use a fourth-order moment system. 
Since higher-order moment systems include more information about the internal state distribution, it is reasonable to expect that they can yield better approximations. In fact, the larger the path-wise gradient is, the wider the distribution of the internal state spreads. Therefore more moments should be included. However, it is not yet fully understood what the correct number of moments one should choose to obtain an accurate approximation. This depends on how fast the environment varies, i.e. how large the space gradient is and how quickly the signal changes.

In this paper, instead of using moment closure in the internal state, we derive both formally and rigorously macroscopic models as asymptotic limits of pathway kinetic equations for E.coli chemotaxis.
We also show numerically that so-derived macroscopic models match quantitatively with the agent-based model {\it for all ranges of chemical gradients}. 

The pathway kinetic model we consider contains both individual bacteria movement by run-and-tumble and an intra-cellular molecular content  \cite{SWOT}.
This equation governs the evolution of the probability density function $p(x,v,m,t)$ of bacteria at time $t$, position $x \in\R^d$, velocity $v \in \VV$, and methylation level $m >0$. In this paper, we will restrict ourselves to the discrete kinetic equation where $x \in \R^1$ and 
\begin{align*}
   \VV = \{-v_0, v_0\} \,,
\qquad 
   \dv = \frac{1}{2} \vpran{\delta(v - v_0) + \delta(v + v_0)} \,.
\end{align*} 
Here $v_0 > 0$ is the fixed speed. The general form of the kinetic equation is
\begin{align} \label{eq:kinetic-m} 
  \del_t p + v \cdot \nabla_x p + \del_m [f(m,M)p] 
  &= Q[m,M](p) \,,
\end{align}
where $M(x,t)$ is the methylation level at equilibrium which relates to the extra-cellular chemical signal.
The function $f(m,M)$ describes the intracellular adaptation dynamics that gives the evolution of the methylation level. The tumbling term $Q[m,M](p)$ satisfies
\begin{align} \label{def:Q}
 Q[m,M](p)
 =\int_\VV  \left[
     \lambda (m, M, v, v') p(t, x, v', m)  
     - \lambda (m, M, v', v) p(t, x, v, m) \right] \dv',
\end{align}
where $\lambda(m, M, v, v')$ denotes the methylation dependent tumbling frequency from $v'$ to $v$, in other words the response of the cell depending on its environment and internal state. The methylation level $M(x,t)$ is related to the extra-cellular attractant profile $S$ by a logarithmic dependency such that
$$M = M(S) = m_0 + \frac{f_0(S)}{\alpha_0}, 
\qquad\mbox{with}\quad
   f_0(S) = \ln \biggl(\frac{1+S/K_I}{1+S/K_A}\biggr) \,.$$
The constant $m_0$ is a reference methylation level in the absence of signal and
the constants $K_I$ and $K_A$ represent the dissociation constants for inactive and active receptors respectively. They satisfy the relation that $K_I \ll S \ll K_A$. Therefore, $f_0(S)  \approx \ln (S/K_I)$. In this paper, we will use
\begin{align*}
   f_0(S) = \ln (S/K_I) \,,
\qquad
   S = S_0e^{\int_0^xG(x')\dx'} \,.
\end{align*}
Then the gradient of $M$ simplifies to 
\begin{align} \label{cond:del-M}
    \del_x M = G/\alpha_0 \,. 
\end{align}
We will consider the the case when $G$ is uniform in space, which is the exponential environment as in the experiment in 
 \cite{kalinin}.

As in \cite{SWOT}, we assume that the tumbling frequency $\lambda$ is independent of $v$ and $v'$. Moreover, the specific forms of the intracellular dynamics and the tumbling frequency are given by 
\begin{align} \label{assump:f-Lambda}
   f \big(m - M\big) 
   = F_0(a)
   = k_R (1 - a/a_0) \,,
\quad 
     \lambda(m,M,v,v')
     = Z(a)
     = z_0 + \tau_0^{-1} \left(\frac{a}{a_0} \right)^H, 
\end{align}
where $a(m-M(S))$ is the receptor activity that
 depends on the intracellular
methylation level $m$ and the extracellular chemoattractant
concentration $S$ in the way that 
\begin{align}  \label{def:a-m}
    a=\bigl(1+\exp(NE)\bigr)^{-1} \,,
\qquad\mbox{with } 
    E=-\alpha_0 (m-m_0)+f_0(S)= - \alpha_0 (m-M(S)) .
\end{align}
Here the coefficient $N$ represents the number of tightly coupled receptors. 
The parameter $k_R$ is the methylation rate, $a_0$ is the receptor preferred activity. The parameters $z_0$, $H$, $\tau_0$ in the tumbling frequency represent the rotational diffusion, the Hill coefficient of flagellar motors response curve, and the average run time respectively. All these parameters can be measured biologically. For more details about the derivation of these formalisms and physical meanings of these parameters, we refer the reader to \cite{SWOT} and the references therein.

The most widely used macroscopic (or population-level) model is the Keller-Segel equation, which was first introduced in \cite{Patlak}. Later, Keller and Segel used it to model chemotaxis behavior of bacteria and cells \cite{KS1,KS2}. It reads
 $$
 \partial_t\rho=\nabla\cdot(D\nabla\rho-\kappa\rho\phi(\nabla S)).
 $$
The fundamental question of how to determine $D$ and $\phi(\nabla S)$ from the molecular mechanisms of chemotaxis has been studied in \cite{SWOT,STY,X} using~\eqref{eq:kinetic-m},
where the molecular origins of the logarithmic sensitivity \cite{kalinin} of the E. coli chemotaxis is justified for slowly varying environment.
However, as pointed out in \cite{SWOT}, both Keller-Segel equation and BPMFT fail to give the right average drift velocity in the exponential environment when the chemical gradient becomes large. 
The valid macroscopic equation that can match quantitatively with the agent-based model for large gradient environment has been open since then and this is what we want to address in this paper. In particular, we give an answer to the question of how to determine the population level drift velocity from the molecular mechanisms of chemotaxis, {\it for all ranges of chemical gradients}. It is shown that for large chemical gradients, the leading-order macroscopic equations are hyperbolic, therefore the diffusion term is of higher order compared with the advection term. When the chemical gradient decreases, the leading-order macroscopic equation becomes the standard Keller-Segel equation with the same diffusion and advection coefficient as in \cite{STY}.

The rest of the paper is organized as follows.
In Section 2, we use asymptotic analysis to formally derive the leading-order macroscopic equations from~\eqref{eq:kinetic-m}.
Quantitative agreement of the distribution function as well as the drift velocity of the agent based simulation and our analytical results are numerically shown in Section 3. In Section 4, we introduce various scalings to~\eqref{eq:kinetic-m} and rigorously show the convergence of the kinetic model to the macroscopic models derived in Section 2. We then conclude in Section 5.

\section{Formal Asymptotics}\label{sec:asymptotics}

In this section we formally derive the leading-order macroscopic equations from the kinetic equation~\eqref{eq:kinetic-m}.
Both the leading order distribution and the chemotaxis drift velocity will be derived explicitly.

Throughout this paper we will use the notation
\begin{align*}
   \vint{\vint{F}}_{v, a} = \int_0^1\int_\VV \f{F}{N\alpha_0a(1-a)}\dv\da \,,
\qquad
   \vint{F}_{v} = \int_\VV F\dv \,,
\end{align*}
for any function $F$ which makes sense of the above integrals. 

We start with reformulating equation~\eqref{eq:kinetic-m}. Since the adaptation rate $f$ is in a particularly simple form in the receptor activity $a$, we re-write equation~\eqref{eq:kinetic-m} in $(t, x, v, a)$. To this end, let $q(t,x, v, a) = p(t, x, v, m)$. Then $q$ satisfies 
\begin{align*}
   \del_t q + v \cdot \nabla_x q 
   + (\del_a q) (D_t M) \frac{\del a}{\del M} 
   + \frac{\del a}{\del m} \del_a (F_0(a) q)
   = Z(a) \CalL q_\Eps\,,
\end{align*}
where $F_0(a)$ is defined in~\eqref{assump:f-Lambda} and by \eqref{def:a-m}, we have
\begin{align} \label{def:L}
 \frac{\del a}{\del M} 
    = - \frac{\del a}{\del m}
    = - N \alpha_0 a (1 - a) \,,\qquad
   \CalL q_\Eps = \int_\VV (q(t, x, v', a) - q(t, x, v, a)) \dv' \,.
\end{align} 
Therefore, the $q$-equation becomes
\beq \label{eq:q}
\begin{aligned}
  & \del_t q + v \cdot \nabla_x q 
   + N \alpha_0 a(1 - a) \del_a \vpran{\vpran{- D_t M + k_R \vpran{1 - \frac{a}{a_0}}} q}
   = Z(a) \CalL q_\Eps \,.
\end{aligned}
\eeq
The weighted average of $q$ in $a$ satisfies
\begin{align*}
    \bar q(t, x, v) 
&    = \int_0^1 q(t, x, v, a) \frac{\del m}{\del a} \da 
    = \int_0^1 q(t, x, v, a) \frac{1}{N \alpha_0 a(1-a)} \da 
\\
&    = \int_{-\infty}^\infty p(t, x, v, m) \dm
    = \bar p(t, x, v) \,.
\end{align*}
%
%
Furthermore, by the assumption \eqref{cond:del-M}, the $q$-equation simplifies to
\begin{align} \label{eq:q-1}
   \del_t q + v \del_x q 
   + N \alpha_0 a(1 - a) \del_a \vpran{\vpran{- \f{v}{\alpha_0} G + k_R \vpran{1 - \frac{a}{a_0}}} q}
   =Z(a) \CalL q_\Eps \,,
\end{align}
with $a_0 = 1/2$ as in \cite{SWOT}. 


To perform the asymptotic analysis, we
introduce the small parameter $\Eps > 0$ and rescale equation~\eqref{eq:q-1} as
\begin{align} \label{eq:q-1-scaled}
   \Eps^\beta \del_t q_\Eps + \Eps v \del_x q_\Eps 
   + N \alpha_0 a(1 - a) \del_a \vpran{\vpran{-\f{ v}{\alpha_0} G + k_R \vpran{1 - \frac{a}{a_0}}} q_\Eps}
   = Z(a) \CalL q_\Eps \,
\end{align}
with $\beta \in [1,2]$. 
Denote the total density and the density flux as $\rho_\Eps$ and $J_\Eps$ such that
\begin{align*}
    \rho_\Eps(x,t)
    =\vint{\vint{q_\epsilon}}_{v,a} \,,
\qquad
    J_\Eps(x,t)
    =\vint{\vint{vq_\epsilon}}_{v,a} \,.
\end{align*}
Then $\rho_\Eps, J_\Eps$ satisfy the macroscopic equation
\begin{align}\label{eq:macro}
    \epsilon^{\beta-1}\p_t\rho_{\epsilon}+ \p_xJ_\epsilon=0 \,.
\end{align}

The main part for the asymptotic analysis is to show how to close equation~\eqref{eq:macro}. The idea is to use the leading-order distribution given by 
the ODE
\begin{align} \label{eq:q-0-1}
   N \alpha_0 a(1 - a) \del_a \vpran{\vpran{- \f{v}{\alpha_0} G + k_R \vpran{1 - \frac{a}{a_0}}} q_0}
   = Z(a) \CalL q_0 \,.
\end{align}
The solution to the above ODE can be found explicitly and we have the following proposition:
\begin{proposition} \label{prop:leading-order}
Suppose the velocity space is discrete such that 
\begin{align*}
    \VV = \{v_0, -v_0\} \,, 
\qquad
    \dv = \frac{1}{2} \vpran{\delta(v - v_0) + \delta(v + v_0)} \,.
\end{align*}
Let $q_0$ be a probability density function and denote
\begin{align} 
   q_0^+ = q_0(t, x, v_0, a) \,,
\qquad
   q_0^- = q_0(t, x, -v_0, a) \,, \nn
\\[2pt]
    g=\frac{v_0}{k_R} \f{G}{\alpha_0}\,,\qquad
      a_1 = \frac{1 - g}{2} \,,
\qquad
   a_2 = \frac{1 + g}{2} \,. \label{def:g-a-1-2}
\end{align}

\Ni (a) If $g > 1$, then $a_1 < 0 < 1 < a_2$. In this case the solution to \eqref{eq:q-0-1}  in the space of probability measures has the form 
\begin{align}
   q_0^+(t, x, a)
   = \rho_0(t, x) c_0 \frac{\frac{1}{2} - a_1}{a - a_1} \, 
      \exp\vpran{\frac{1}{4 N \alpha_0 k_R} 
           \int_{1/2}^a \frac{Z(\tau)}{\tau (1-\tau)} \frac{2\tau - 1}{(a_1 - \tau) (a_2 - \tau)} \dtau} \,
            \Denote \rho_0 Q_0^+ \,,
    \label{soln:q-0-plus-1}
\\
    q_0^-(t, x, a)
   = \rho_0(t, x) c_0 \frac{\frac{1}{2} - a_1}{a_2 - a} \, 
      \exp\vpran{\frac{1}{4 N \alpha_0 k_R} 
           \int_{1/2}^a \frac{Z(\tau)}{\tau (1-\tau)} \frac{2\tau - 1}{(a_1 - \tau) (a_2 - \tau)} \dtau} \,
           \Denote \rho_0 Q_0^- \,.
    \label{soln:q-0-minus-1}
\end{align}
where $\rho_0$ is a probability measure  and $c_0 > 0$ is determined by the normalization condition
\begin{align} \label{cond:normalization}
    \int_\R \int_0^1 \frac{q^+_0 + q^-_0}{2N \alpha_0 a(1-a)} \da \dx = 1 \,.
\end{align}

\Ni (b) If $0 < g < 1$, then $0 < a_1 < \frac{1}{2} < a_2 < 1$. In this case the solution to \eqref{eq:q-0-1}  in the space of probability measures has the same form as in~\eqref{soln:q-0-plus-1} and~\eqref{soln:q-0-minus-1} for $a \in (a_1, a_2)$ and
\begin{align*}
    q_0^+ = q_0^- = 0 \,,
\qquad
    a \in [0, a_1) \cup (a_2, 1] \,.
\end{align*}

\Ni (c) If $g=0$, then the solution to \eqref{eq:q-0-1}  in the space of probability measures is 
\begin{align*}
    q_0^+ = q_0^- = \frac{N \alpha_0}{4} \rho_0(t, x) \delta_{1/2}(a) \,, 
\end{align*}
where $\rho_0$ is a probability measure such that the normalization condition~\eqref{cond:normalization} holds. 
\end{proposition}
Before showing the details of the proof of the above proposition, we derive the formal closures for~\eqref{eq:macro} by using Proposition \ref{prop:leading-order}.
\subsection{Formal Asymptotic Limits}
We will divide the analysis according to $g = \BigO(1)$ and $g= o(1)$. The difference between these two ranges is that in the former case, we only the leading-order distribution
is used, while in the latter case we need to use the next-order correction as well. 

\subsection*{Case I: $g = \BigO(1)$}
In this case we formally decompose $q_\Eps$ according to the orders of $\Eps$ such that
\begin{align*}
    q_\Eps = q_0 + \Eps q_1 + \cdots \,.
\end{align*}
 By matching the terms in~\eqref{eq:q-1-scaled}, we derive that the leading-order term $q_0$ satisfies the ODE~\eqref{eq:q-0-1}. 
Let $\beta=1$ and close equation~\eqref{eq:macro} by its leading-order approximation $q \approx q_0$. Then \eqref{eq:macro} becomes
\begin{align*}
   \del_t \int_0^1 \frac{q_0^+ + q_0^-}{2N \alpha_0 a(1-a)} \da 
   +  \del_x \int_0^1 \frac{v_0(q_0^+ - q_0^-)}{2N \alpha_0 a(1-a)} \da
  = 0 \,.
\end{align*}
Therefore, from \eqref{soln:q-0-plus-1}, \eqref{soln:q-0-minus-1}, $\rho_0$ satisfies the transport equation
\begin{align} \label{eq:macro-1}
   \del_t \rho_0
   + \del_x  (\kappa_1 \rho_0)  = 0 \,,
\quad
   \text{for $g > 1$} \,,
\end{align}
and 
\begin{align} \label{eq:macro-2}
   \del_t \rho_0
   + \del_x  (\kappa_2 \rho_0)  = 0 \,,
\quad
   \text{for $0 < g < 1$} \,,
\end{align}
where if $g > 1$ then the transport speed is
\begin{align}
   \kappa_1 
&   = v_0 \vpran{\int_0^1\frac{Q_0^+ - Q_0^-}{a(1-a)} \da}
\Big/
\vpran{\int_0^1\frac{Q_0^+ + Q_0^-}{a(1-a)} \da}
\label{eq:kappa1}
\\
 &  = \frac{v_0}{a_2-a_1} \int_0^1 \frac{1 - 2a}{a(1-a)} 
      \exp\vpran{\frac{1}{4 N \alpha_0 k_R} 
           \int_{1/2}^a \frac{Z(\tau)}{\tau (1-\tau)} \frac{2\tau - 1}{(a_1 - \tau) (a_2 - \tau)} \dtau} \da \,,\nonumber
\end{align}
and if $0 < g < 1$, then the transport speed is
\begin{align}
   \kappa_2 
&  = v_0 \vpran{\int_{a_1}^{a_2}\frac{Q_0^+ - Q_0^-}{a(1-a)} \da}
\Big/
\vpran{\int_{a_1}^{a_2}\frac{Q_0^+}{a(1-a)} \da}
\label{eq:kappa2}
\\
 & =\frac{v_0}{a_2-a_1} \int_{a_1}^{a_2} \frac{1 - 2a}{a(1-a)} 
      \exp\vpran{\frac{1}{4 N \alpha_0 k_R} 
           \int_{1/2}^a \frac{Z(\tau)}{\tau (1-\tau)} \frac{2\tau - 1}{(a_1 - \tau) (a_2 - \tau)} \dtau} \da \,.\nonumber
\end{align}
Here $a_1, a_2$ are defined in~\eqref{def:g-a-1-2} and $Q_0^\pm$ are defined in~\eqref{soln:q-0-plus-1}-\eqref{soln:q-0-minus-1}. 

\subsection*{Case II: $g = o(1)$} 
%
%
%
%
%
To be precise, we consider the case where $g = \BigO(\Eps^\mu)$ with $0<\mu \leq 1$. Let
\begin{align}\label{eq:Gmu}
    G_\mu = G /\Eps^\mu = \BigO(1) \,,
\qquad
    g_\mu = \frac{v_0G_\mu}{k_R\alpha_0} =  \BigO(1) \,.
\end{align}
Equation~\eqref{eq:q-1-scaled} becomes 
\begin{align} \label{eq:q-3-scaled}
   \Eps^{1+\mu} \del_t q_\Eps + \Eps v \del_x q_\Eps 
   + N \alpha_0 a(1 - a) \del_a \vpran{\vpran{- \Eps^\mu v  \f{G_\mu}{\alpha_0} + k_R \vpran{1 - 2a}} q_\Eps}
   = Z(a) \CalL q_\Eps \,,
\end{align}
where $\CalL$ is defined in~\eqref{def:L}. Formally decompose $q_\Eps$ as
\begin{align}\label{eq:qepsmu}
    q_\Eps = q_0 + \Eps^\mu q_1 + o(\Eps^\mu) \,.
\end{align}
Then the leading-order equation is when $g=0$ in \eqref{eq:q-0-1}, which yields 
\begin{align} \label{soln:q-0-case-3}
    q_0(t, x, v, a) = \frac{N \alpha_0}{4}\rho_0(t, x) \delta_{1/2}(a) .
\end{align}
Meanwhile, if we denote the first few orders of $q_\Eps$ up to $\BigO(\Eps)$ as $\tilde q_\Eps$, then $\tilde q_\Eps$ satisfies the equation
\begin{align} \label{eq:tilde-q-Eps-2}
 \Eps v \del_x \tilde q_\Eps  + N \alpha_0 a(1 - a) \del_a \vpran{\vpran{- v \Eps^\mu G_\mu + k_R \vpran{1 - 2a}} \tilde q_\Eps}
   = Z(a) \CalL \tilde q_\Eps \,.
\end{align}
For any fixed $\Eps$ small enough, we have $0< a_1 < 1/2 < a_2 < 1$. Thus $\tilde q_\Eps$ is compactly supported on $[a_1, a_2] = [1/2 - \Eps^\mu g_\mu, 1/2 + \Eps^\mu g_\mu]$. 
Therefore, we rewrite equation~\eqref{eq:tilde-q-Eps-2} as
\begin{align} \label{eq:q-2-scaled-1}
   \Eps v \del_x \tilde q_\Eps
   + N \alpha_0 a(1 - a) \del_a \vpran{\vpran{- v \Eps^\mu G_\mu 
      + k_R \Eps^\mu \frac{1-2a}{\Eps^\mu}} \tilde q_\Eps}
   = Z(a) \CalL \tilde q_\Eps \,,
\end{align}
where the term $\Disp\frac{1-2a}{\Eps^\mu}$ is uniformly bounded in $\Eps$. 

Now we separate the two cases where $0 < \mu < 1$ and $\mu = 1$. 

\medskip

\Ni $\bullet$ First, if $0 < \mu < 1$, then by matching the terms at the leading order in~\eqref{eq:q-2-scaled-1}, we obtain the equation for $q_1$ as
\begin{align*}
      N \alpha_0 a(1 - a) \del_a \vpran{\vpran{- v G_\mu 
      + k_R \frac{1-2a}{\Eps^\mu}} q_0}
   = Z(a) \CalL q_1 \,.
\end{align*}
By \eqref{soln:q-0-case-3} this simplifies to 
\begin{align} \label{eq:q-1-case-3-1}
    - \frac{v G_\mu}{Z(a)} N \alpha_0 a(1 - a) \del_a q_0(t, x, a)
   = \CalL q_1 \,.
\end{align}
Since the desired term is the flux term $J_\Eps = \vint{\vint{vq_1}}_{v,a}$, we multiply $v$ to equation~\eqref{eq:q-1-case-3-1} and integrate in $v$. This gives
\begin{align} \label{eq:q-1-case-3}
    - \frac{v_0^2 G_\mu}{Z(a)} N \alpha_0 a(1 - a) \del_a q_0(t, x, a)
   = \vint{v q_1}_v \,.
\end{align}
Therefore the flux term is computed as 
\begin{align*}
   \int_0^1\vint{v q_1}_v \frac{1}{N \alpha_0 a(1-a)} \da
   = -\int_0^1 \frac{v_0^2 G_\mu}{Z(a)} \del_a 
        q_0 \da
   = \frac{N\alpha_0}{4}\rho_0(t, x) v_0^2 G_\mu \vpran{\frac{1}{Z(a)}}' \Big|_{a = 1/2} \,.
\end{align*}
Let
\begin{align} \label{def:kappa-3}
    \kappa_3 = \frac{N\alpha_0}{4}v_0^2 G_\mu \vpran{\frac{1}{Z(a)}}' \Big|_{a = 1/2} \,.
\end{align}
Then the moment closure has the form
\begin{align} \label{eq:macro-2}
    \del_t \rho_0 + \del_x (\kappa_3 \rho_0) = 0 \,.
\end{align}

\Ni $\bullet$ Now we consider the case where $\mu = 1$. In this case the only difference is equation~\eqref{eq:q-1-case-3} has an addition term from the advection and the new equation is 
\begin{align} \label{eq:q-1-case-4}
   \frac{v_0^2}{Z(a)} \del_x q_0 
    - \frac{v_0^2 G_1}{Z(a)} N \alpha_0 a(1 - a) \del_a q_0(t, x, a)
   = \vint{v q_1}_v \,,
\qquad
   G_1 = G/\Eps \,.
\end{align}
%
Integrating in $a$ gives the flux term as
\begin{align*}
& \quad \,
   \int_0^1\vint{v q_1}_v \frac{1}{N \alpha_0 a(1-a)} \da
\\
& = \frac{N\alpha_0}{4}(v_0^2 \del_x \rho_0) \int_0^1 \frac{1}{Z(a)} \delta_{1/2}(a) \frac{1}{N \alpha_0 a(1-a)} \da
   - v_0^2 G_1 \frac{N\alpha_0}{4}\int_0^1 \frac{1}{Z(a)} \del_a 
   q_0 \, \da
\\
& = \frac{v_0^2}{Z(1/2)} \del_x \rho_0
     + \frac{N\alpha_0}{4}\rho_0(t, x) v_0^2 G_1 \vpran{\frac{1}{Z(a)}}' \Big|_{a = 1/2} \,.
\end{align*}
Then the moment closure is the classical Keller-Segel equation:
\begin{align} \label{eq:KS}
   \del_t \rho_0
   + \del_{x} (D_0 \del_x\rho_0)
   +  \del_x (\kappa_3\rho_0) = 0 \,,
\end{align}
where the coefficient $D_0$ is
\begin{align*}
   D_0 = \frac{v_0^2}{Z(1/2)} = \frac{v_0^2}{z_0 + \tau_0^{-1}} \,,
\end{align*}
and $\kappa_3$ is the same transport speed defined in~\eqref{def:kappa-3}.
\begin{remark}
Since the distributions in $a$ for those forward and backward moving bacteria are explicitly known in \eqref{soln:q-0-plus-1}, \eqref{soln:q-0-minus-1},
when $g=\BigO(\Eps^\mu)$, $\mu\in(0,1)$,
we can not only get the leading order distribution in \eqref{soln:q-0-case-3}, but also the distribution up to $\BigO(\Eps^\mu)$. Therefore, the macroscopic equation up to $\BigO(\Eps^{1-\mu})$
can be obtained as well. An additional $\BigO(\Eps^{1-\mu})$ diffusion term will appear in the macroscopic equation which formally tends the diffusion in the Keller-Segel model 
when $\mu\to 1$. 
\end{remark}

\subsection{The leading order distribution}
The solution of \eqref{eq:q-0-1} plays an essential role in the derivation of the macroscopic equation, we prove Proposition
\ref{prop:leading-order} and show some properties of leading order distribution in this part. 
\begin{proof}[Proof of Proposition 2.1]
First note that since $q_0^\pm$ are finite measures, by~\eqref{eq:q-0-1} we have
\begin{align*}
     \text{$(a_1 - a) q_0^+$ and $(a_2 - a) q_0^-$ are both BV functions on any $(c, d) \subsetneq (0, 1)$} \,.
\end{align*}
Moreover, by the normalization condition~\eqref{cond:normalization}, 
\begin{align*}
    \liminf_{a \to 0} q_0^\pm(a) 
    = \liminf_{a \to 1} q_0^\pm(a)
    = 0 \,.
\end{align*}

\Ni (a) Since the velocity space $\VV$ is discrete, we have
\begin{align*}
    \vpran{\CalL q_0}^+(t, x, a) = \frac{1}{2} \vpran{q_0^- - q_0^+} \,,
\qquad
    \vpran{\CalL q_0}^-(t, x, a) = \frac{1}{2} \vpran{q_0^+ - q_0^-} \,.
\end{align*}
Together with the notation introduced in~\eqref{def:g-a-1-2}, equations~\eqref{soln:q-0-plus-1}-\eqref{soln:q-0-minus-1} become
\begin{align}
   a(1 - a) \del_a \vpran{(a_1 - a) q_0^+}
   = \frac{Z(a)}{4 N \alpha_0 k_R} \vpran{q_0^- - q_0^+} \,,
   \label{eq:q-0-plus}
\\
   a(1 - a) \del_a \vpran{(a_2 - a) q_0^-}
   = \frac{Z(a)}{4 N \alpha_0 k_R} \vpran{q_0^+ - q_0^-} \,.
   \label{eq:q-0-minus}
\end{align}
This is a system of two ODEs which we can solve explicitly. 
Since the variables $t, x$ do not appear explicitly  in equations~\eqref{eq:q-0-plus}-\eqref{eq:q-0-minus}, the solution $q_0$ will be in a separated form such that
\begin{align} \label{soln:q-0-separated}
    q_0(t, x, v, a) = \rho_0(t, x) \, Q_0(v, a) \,,
\qquad
    q_0^\pm(t, x, a) = \rho_0(t, x) \, Q_0^\pm (a) \,,
\end{align}
where $Q_0^\pm$ satisfy the normalization condition
\begin{align} \label{cond:normalization-1}
    \int_0^1 \frac{Q_0^+ + Q_0^-}{2N \alpha_0 a(1-a)} \da = 1  \,.
\end{align}
To solve~\eqref{eq:q-0-plus}-\eqref{eq:q-0-minus}, we add these two equations up and get
\begin{align*}
   a(1 - a) \del_a \vpran{(a_1 - a) Q_0^+ + (a_2 - a) Q_0^-}
   = 0 \,.
\end{align*}
Hence there exists a constant $c_1$ such that for $a \in (0, 1)$,
\begin{align} \label{eq:q-0-1-0-1}
    (a_1 - a) Q_0^+ + (a_2 - a) Q_0^-  = c_1 \,.
\end{align}
Now we show that $c_1 = 0$. Suppose instead $c_1 > 0$. Since $Q_0^\pm$ are both non-negative measures, we have 
\begin{align*}
    Q_0^-(a) \geq \frac{c_1}{a_2 - a} \geq \frac{c_1}{a_2} \geq 0 \,.
\end{align*}
This contradicts the finiteness of $Q_0^-$ in \eqref{cond:normalization-1}. Similarly, if $c_1 < 0$, then
\begin{align*}
    \bar q_0^+(a) \geq \frac{-c_1}{a - a_1} \geq \frac{-c_1}{1- a_1} \geq 0 \,,
\end{align*}
which also contradicts \eqref{cond:normalization-1}. Therefore $c_1 = 0$ and $Q_0^\pm$ satisfy
\begin{align} \label{eq:q-0-1-0-2}
    (a_1 - a) Q_0^+ + (a_2 - a) Q_0^-  = 0  \,.
\end{align}
This gives
\begin{align} \label{eq:q-0-1-1}
    Q_0^- = \frac{a - a_1}{a_2 - a} Q_0^+  \,.
\end{align}
Applying~\eqref{eq:q-0-1-1} in~\eqref{eq:q-0-plus}, we get for $a \in (0, 1)$, 
\begin{align*}
   \del_a \vpran{(a_1 - a) Q_0^+}
   = \frac{1}{4 N \alpha_0 k_R} 
      \frac{Z(a)}{a(1-a)} \frac{2a - 1}{(a_1 - a) (a_2 - a)} (a_1 - a) Q_0^+ \,.
\end{align*}
Solving this ODE for $(a_1 - a) \bar q_0^+$ gives
\begin{align*}
   (a_1 - a) Q_0^+
   = c_0 (a_1 - 1/2) \, 
      \exp\vpran{\frac{1}{4 N \alpha_0 k_R} 
           \int_{1/2}^a \frac{Z(\tau)}{\tau (1-\tau)} \frac{2\tau - 1}{(a_1 - \tau) (a_2 - \tau)} \dtau} \,.
\end{align*}
This combined with \eqref{eq:q-0-1-1} gives~\eqref{soln:q-0-plus-1} and~\eqref{soln:q-0-minus-1}. Note that we do not have concentration at $a = a_1, a_2$ since $a_1, a_2 \notin [0, 1]$.

\medskip

\Ni (b) The proof of (b) is similar to (a). Note that both $(a_1 - a) Q_0^+$ and $(a_2 - a) Q_0^-$ are again in $BV(c, d)$ for any $(c, d) \subsetneq (0, 1)$. Thus
equations~\eqref{eq:q-0-plus}-\eqref{eq:q-0-minus} and ~\eqref{eq:q-0-1-0-1} still holds on $(0, 1)$. Now we show $c_1 = 0$ when $0 < g <1$. In this case, $0 < a_1 < 1/2 < a_2 < 1$. Since $Q_0^\pm$ are non-negative measures, we have
\begin{align*} 
    (a_1 - a) Q_0^+ + (a_2 - a) Q_0^-  \geq 0 \,,
\qquad a \in (0, a_1) \,,
\\
    (a_1 - a) Q_0^+ + (a_2 - a) Q_0^-  \leq 0 \,,
\qquad a \in (a_2, 1) \,.
\end{align*}
Therefore $c_1 = 0$ and \eqref{eq:q-0-1-0-2} holds. This also implies
\begin{align*}
    Q_0^+(a) = Q_0^-(a) = 0 \,,
\qquad
   a \in (0, a_1) \cup (a_2, 1) \,.
\end{align*}
Thus $Q_0^\pm$ are compactly supported on $[a_1, a_2]$. Solving~\eqref{eq:q-0-1-0-2} gives
\begin{align} \label{eq:q-0-1-1-b}
    Q_0^- 
    = \frac{a - a_1}{a_2 - a} Q_0^+ 
       + c_2 \delta_{a_2}(a) \,,
\qquad 
    a \in [a_1, a_2] \,.
\end{align}
for some constant $c_2$. Solving~\eqref{eq:q-0-plus} on $(a_1, a_2)$ then gives~\eqref{soln:q-0-plus-1}-\eqref{soln:q-0-minus-1} on $(a_1, a_2)$ where $c_0 > 0$ may not satisfy the normalization condition since there can be concentration of $q_0^+$ at $a_1$ and $q_0^+$ at $a_2$. Now we show that there cannot be such concentrations. This is because both $(a_1 - a) q_0^+$ and $(a_2 - a) q_0^-$ are Lipschitz. Thus by~\eqref{eq:q-0-plus}-\eqref{eq:q-0-minus}, $q_0^- - q_0^+$ is in $L^\infty$. Hence they cannot have concentrations at $a_2, a_1$ respectively. Therefore, the solution to~\eqref{eq:q-0-1} for $0 < g < 1$ is as claimed in part (b).

\medskip

\Ni (c) If $g=0$, then $a_1 = a_2 = \frac{1}{2}$. In this case equation~\eqref{eq:q-0-1} becomes
\begin{align} \label{eq:q-0-1-c}
   N \alpha_0 k_R a(1 - a) \del_a \vpran{\vpran{1 - 2a} q_0}
   = Z(a) \CalL q_0 \,,
\end{align} 
or equations~\eqref{eq:q-0-plus}-\eqref{eq:q-0-minus} become
\begin{align} 
   a(1 - a) \del_a \vpran{(1 - 2a) q_0^+}
   = \frac{Z(a)}{4 N \alpha_0 k_R} \vpran{q_0^- - q_0^+} \,,
   \label{eq:q-0-plus-g-0}
\\
   a(1 - a) \del_a \vpran{(1 - 2a) q_0^-}
   = \frac{Z(a)}{4 N \alpha_0 k_R} \vpran{q_0^+ - q_0^-} \,.
   \label{eq:q-0-minus-g-0}
\end{align}
Adding up equations~\eqref{eq:q-0-plus-g-0}-\eqref{eq:q-0-minus-g-0} gives
\begin{align*} 
   N \alpha_0 k_R \, a(1 - a) \del_a \vpran{\vpran{1 - 2a} \vpran{q_0^+ + q_0^-}}
   = 0 \,.
\end{align*}
Therefore, 
\begin{align*}
   (1 - 2a) \vpran{q_0^+ + q_0^-} =  c_3 \,,
\qquad
   a \in (0, 1) \,,
\end{align*}
for some constant $c_3$. Since $q_0^+ + q_0^-$ is a non-negative measure, the only possible choice for $c_3$ is $c_3 = 0$. Hence
\begin{align*}
    \vint{q_0}_v = \frac{q_0^+ + q_0^-}{2} 
    = N \alpha_0 a (1 - a) \rho_0(t, x) \delta_{1/2}(a)
    = \frac{N \alpha_0}{4} \rho_0(t, x) \delta_{1/2}(a) \,,
\end{align*}
where $\rho_0$ is a probability measure. This implies that  
\begin{align*}
    (1- 2a)\CalL q_0 = - (1-2a) q_0 \,.
\end{align*}
Thus multiplying~\eqref{eq:q-0-1-c} by $(1 - 2a)$ gives
\begin{align}  \label{eq:q-0-2-1}
   N \alpha_0 k_R (1 - 2a) a(1 - a) \del_a \vpran{\vpran{1 - 2a} q_0}
   = - Z(a) (1 - 2a) q_0  \,,
\end{align}
By the finiteness of the measure $q_0$ as defined in~\eqref{cond:normalization}, the boundary conditions of $(1 - 2a) q_0$ are 
\begin{align} \label{cond:bdry-c}
   (1 - 2a) q_0 = 0 \,,
\qquad
   a = 0, 1/2, 1 \,.
\end{align}
Note that equation~\eqref{eq:q-0-2-1} shows 
\begin{align*}
  \del_a ((1-2a) q_0) 
  \begin{cases}
     \leq 0 \,, & a < 1/2 \,, \\[2pt]
     \geq 0 \,, & a > 1/2 \,.
  \end{cases}
\end{align*}
Combined with the boundary conditions in~\eqref{cond:bdry-c}, equation~\eqref{eq:q-0-2-1} has a unique solution (up to multiplication by $\rho_0(t, x)$) such that $(1 - 2a) q_0 = 0$. Thus the only solution to~\eqref{eq:q-0-1-c} is 
\begin{align*} 
    q_0 = \vint{q_0}_v 
    = N \alpha_0 a(1-a)\rho_0(t, x) \delta_{1/2}(a)
    = \frac{N \alpha_0}{4} \rho_0(t, x) \delta_{1/2}(a) \,.
\end{align*}

%
%
%
%
%
%
\end{proof}

We can also study in more details the behaviour of the solution $q_0$ in Proposition~\ref{prop:leading-order} near $0, 1$ when $g > 1$ and near $a_1, a_2$ when $0 < g < 1$.
\begin{lem}
Let $q_0$ be the solution to the ODE~\eqref{eq:q-0-1} in Proposition~\ref{prop:leading-order} and $q_0(t, x, v, a) = \rho_0(t, x) Q_0(v, a)$.

\Ni (a) If $g > 1$, then there exists $\theta_0, \theta_1 > 0$ such that 
\begin{align*}
    Q_0^\pm (a) = \BigO\vpran{a^{\theta_0}} \,\, \text{for $a$ near 0} \,,
\qquad
    Q_0^\pm (a) = \BigO\vpran{a^{\theta_1}} \,\, \text{for $a$ near 1} \,.
\end{align*}

\Ni (b) If $0 < g < 1$, then there exists $\theta_2, \theta_3 > 0$ such that
\begin{align*}
   Q_0^+(a) 
   = \begin{cases}
        \BigO\vpran{(a - a_1)^{\theta_2 - 1}} \,, & a \to a_1^+ \,, \\[2pt]
        \BigO\vpran{(a_2 - a)^{\theta_3}} \,, & a \to a_2^- \,,
      \end{cases}
\qquad \quad 
   Q_0^-(a)
   = \begin{cases}
        \BigO\vpran{(a - a_1)^{\theta_2}} \,, & a \to a_1^+ \,, \\[2pt]
        \BigO\vpran{(a_2 - a)^{\theta_3 - 1}} \,, & a \to a_2^- \,.
      \end{cases}
\end{align*}
\end{lem}
\begin{proof}
(a)  By the definition of $q_0^+$ in~\eqref{soln:q-0-plus-1},  the asymptotic limit of $Q_0^+$ satisfies
\begin{align*}
   \lim_{a \to 0} Q_0^+
 &  = c_0 \frac{\frac{1}{2} - a_1}{-a_1}
        \exp\vpran{-\int_0^{1/2} \frac{A_0(\tau) - A_0(0)}{\tau} \dtau}
      \lim_{a \to 0} 
      \exp\vpran{-\frac{z_0}{4 N \alpha_0 k_R}
       \frac{1}{a_1a_2} \int_{1/2}^a \frac{1}{\tau} \dtau }
\\
  & = c_0 \frac{\frac{1}{2} - a_1}{-a_1} 
        \exp\vpran{-\int_0^{1/2} \frac{A_0(\tau) - A_0(0)}{\tau} \dtau}
        2^{\theta_0} 
        \lim_{a \to 0} a^{\theta_0} \,,
\end{align*}
where
\begin{align*}
     A_0(\tau) 
     = \frac{1}{4 N \alpha_0 k_R} \frac{Z(\tau)(2\tau - 1)}{(1 - \tau) (a_1-\tau)(a_2-\tau)} \,,
\qquad
   \theta_0 = A_0(0) = -\frac{z_0}{4 N \alpha_0 k_R} \frac{1}{a_1a_2} \,.
\end{align*}
The integral involving $A_0(\tau)$ converges at $\tau =0$ since $A_0 \in C^1([0, 1/2])$. Moreover, $\theta_0 > 0$ since $a_1 < 0 < a_2$ and $z_0 > 0$. This shows $Q_0^+$, as well as $Q_0^-$, decays to zero algebraically at $a = 0$. Note that since $z_0$ is generally small, the rate of decay can be sublinear.

Similarly, near $a = 1$ the asymptotic limit of $Q_0^+$ is
\begin{align*}
& \quad \,
   \lim_{a \to 1} Q_0^+
\\
 &  = c_0 \frac{\frac{1}{2} - a_1}{1-a_1}
      \exp\vpran{\int^1_{1/2} \frac{A_1(\tau) - A_1(1)}{1-\tau} \dtau}
      \lim_{a \to 1} 
      \exp\vpran{-\frac{2^H}{4 N \alpha_0 k_R}
       \frac{1}{(1 - a_1) (a_2 - 1)} \int_{1/2}^a \frac{1}{1 - \tau} \dtau }
\\
  & = c_0 \frac{\frac{1}{2} - a_1}{1 - a_1} 
        \exp\vpran{\int^1_{1/2} \frac{A_1(\tau) - A_1(1)}{1-\tau} \dtau}
         2^{\theta_1}
        \lim_{a \to 1} (1 - a)^{\theta_1} \,,
\end{align*}
where 
\begin{align*}
   A_1(\tau) 
   =  \frac{1}{4 N \alpha_0 k_R} \frac{Z(\tau) (2\tau - 1)}{\tau (a_1 - \tau) (a_2 - \tau)} \,,
\qquad
   \theta_1 = A_1(1) = \frac{2^H}{4 N \alpha_0 k_R} \frac{1}{(1 - a_1) (a_2 - 1)} \,.
\end{align*}
Again since $A_1 \in C^1([1/2, 1])$, the integral involving $A_1$ converges at $\tau = 1$. We also have an algebraic decay to zero for $Q_0^\pm$ as $a \to 1$. In this case since $H$ is generally large (for example $H = 10$ in our numerical example), the decay near $a = 1$ is nearly exponential. 

\medskip
\Ni (b) Similar as in part (a), we have near $a_1$ the asymptotic limit of $Q_0^+$ is 
\begin{align*}
& \quad \,
   \lim_{a \to a_1^+} Q_0^+
\\
 &  = c_0 
        \exp\vpran{-\int^{1/2}_{a_1} \frac{A_2(\tau) - A_2(a_1)}{a_1 - \tau} \dtau} 
         \lim_{a \to a_1^+} 
         \frac{\frac{1}{2} - a_1}{a - a_1}
            \exp\vpran{-\frac{Z(a_1)}{4 N \alpha_0 k_R}
       \frac{1}{a_1(1-a_1)} \int_{1/2}^a \frac{1}{a_1 - \tau} \dtau }
\\
  & = c_0 \vpran{\frac{1}{2} - a_1} 
        \exp\vpran{-\int^{1/2}_{a_1} \frac{A_2(\tau) - A_2(a_1)}{a_1 - \tau} \dtau}
        \vpran{\frac{2 k_R}{v_0 G}}^{\theta_2} 
        \lim_{a \to a_1^+} (a-a_1)^{\theta_2 - 1} \,,
\end{align*}
where
\begin{align*}
   A_2(\tau) = \frac{1}{4 N \alpha_0 k_R} \frac{Z(\tau)(2\tau - 1)}{\tau (1-\tau)(a_2 - \tau)} \,,
\qquad
   \theta_2 = A_2(a_1) =  \frac{Z(a_1)}{4 N \alpha_0 k_R} \frac{1}{a_1(1-a_1)} \,.
\end{align*}
The integral involving $A_2$ converges at $a=a_1$ since $A_2 \in C^1([a_1, a_2])$.
Then by~\eqref{soln:q-0-minus-1}, 
\begin{align*}
   \lim_{a \to a_1^+} Q_0^-
   = c_0 \frac{\frac{1}{2} - a_1}{a_2 - a_1} 
       \exp\vpran{-\int^{1/2}_{a_1} \frac{A_2(\tau) - A_2(a_1)}{a_1 - \tau} \dtau}
        \vpran{\frac{2 k_R}{v_0 G}}^{\theta_2}
        \lim_{a \to a_1^+} (a-a_1)^{\theta_2}
    \,.
\end{align*}

Similarly, near $a_2$, we have
\begin{align*}
   \lim_{a \to a_2^-} Q_0^+
 &  = c_0 
         \exp\vpran{\int_{1/2}^{a_2} \frac{A_3(\tau) - A_3(a_2)}{a_2 - \tau} \dtau}
         \frac{\frac{1}{2} - a_1}{a_2 - a_1}
         \lim_{a \to a_2^-} 
            \exp\vpran{-\frac{Z(a_2)}{4 N \alpha_0 k_R}
       \frac{1}{a_2(1-a_2)} \int_{1/2}^a \frac{1}{a_2 - \tau} \dtau }
\\
  & = c_0
        \exp\vpran{\int_{1/2}^{a_2} \frac{A_3(\tau) - A_3(a_2)}{a_2 - \tau} \dtau} 
        \frac{\frac{1}{2} - a_1}{a_2 - a_1} 
        \vpran{\frac{2 k_R}{v_0 G}}^{\theta_3}
        \lim_{a \to a_2^-} (a_2 - a)^{\theta_3}
      \,,
\end{align*}
where
\begin{align*}
   A_3(\tau) = \frac{1}{4 N \alpha_0 k_R} \frac{Z(\tau)(2\tau - 1)}{\tau (1-\tau)(a_1 - \tau)} \,,
\qquad
   \theta_3 = A_3(a_2) = \frac{Z(a_2)}{4 N \alpha_0 k_R} \frac{1}{a_2(1-a_2)} \,.
\end{align*}
Again the integral involving $A_3$ converges at $a=a_2$ since $A_3 \in C^1([a_1, a_2])$. Using~\eqref{soln:q-0-minus-1} again we have
\begin{align*}
   \lim_{a \to a_2^-} Q_0^-
   = c_0 
       \exp\vpran{\int_{1/2}^{a_2} \frac{A_3(\tau) - A_3(a_2)}{a_2 - \tau} \dtau} 
       \vpran{\frac{1}{2} - a_1} 
       \vpran{\frac{2 k_R}{v_0 G}}^{\theta_3}
        \lim_{a \to a_2^-} (a_2-a)^{\theta_3 - 1}
     \,.
\end{align*}
\end{proof}

\begin{remark} Here the value of $\theta_k$ ($0 \leq k \leq 3$) determines the behaviour of $q_0^\pm$ ($q_0^-$) near $a= 0, 1$ or $a=a_1, a_2$. Since $\theta_k>0$ for all $0 \leq k \leq 3$, the integrability in the normalization condition
 \eqref{cond:normalization-1} is guaranteed. We note the following differences between $g > 1$ and $0 < g < 1$:
\begin{itemize}
\item If $g>1$, then $q_0^\pm \to 0$ algebraically as $a \to 0$ or $a \to 1$. The decay rate is given by $\theta_0$ near $a=0$ and $\theta_1$ near $a=1$.

\item If $0 < g < 1$, then $0 < a_1 < 1/2 < a_2 < 1$. In this case 
\begin{align*}
    Z(a_1) &= z_0 + \tau_0^{-1} (a_1/a_0)^{H} = z_0 + \tau_0^{-1} (2a_1)^{H} \,,
\\
     Z(a_2) &= z_0 + \tau_0^{-1} (a_2/a_0)^H =  z_0 + \tau_0^{-1} (2a_2)^{H} \,.
\end{align*}
Thus depending on the values of $z_0, \tau_0$ and $H$, the parameter $\theta_2$ can be less than 1 for some $a_1\in(0,a_0)$, in which case we have $\lim_{a \to a_1^+} q_0^+ = \infty$ with the growth rate $1 - \theta_2$. However, when $a_2$ is close to $0$, by its definition $\theta_2$ can increase to be larger than $1$. Then $\lim_{a \to a_1^+} q_0^+ = 0$. 
On the other hand, since $H$ is large, the parameter $\theta_3$ is more likely to be larger than $1$. 

Using the particular physical parameters for wild type E.coli in section \ref{sec:numerics}, we do have $\theta_2,\theta_3 >~1$ for all 
 $a_1\in(0,1/2)$ and $a_2\in (1/2,1)$. Hence in Section 3 we have algebraic decay of $Q_0^\pm$ near both $a_1$ and $a_2$.
 \end{itemize}

\end{remark}

\section{Comparison with Numerics}\label{sec:numerics}
In this section we specify various types of scalings and compare the numerical results using the agent-based model SPECS and the closures derived in Section~\ref{sec:asymptotics}. 
Recall that the intracellular dynamics and tumbling frequency are given by 
\begin{align*}
   f \big(m - M(S) \big) 
   = F_0(a)
   = k_R (1 - a/a_0) \,,
\qquad 
     \Lambda \big(m - M(S) \big)
     = Z(a)
     = z_0 + \tau_0^{-1} \left(\frac{a}{a_0} \right)^H, 
\end{align*}
where $a(m-M(S))$ is the receptor activity defined in~\eqref{def:a-m}.
The parameters are chosen as in \cite{SWOT} such that
\begin{align*}
    &v_0=\frac{16.5}{\sqrt{2}}\mu m/s,\quad k_R=0.01s^{-1}\sim 0.0005s^{-1},\quad a_0=0.5,\\
    &\alpha_0=1.7,\quad z_0=0.14s^{-1}, \quad\tau_0=0.8s,\quad H=10.
\end{align*}
The external signal is given by $S=S_0\exp(Gx)$, where $G$ takes the values  $0\sim 2*10^{-3}\mu m^{-1}$. Since $f_0(S)$ can be approximated by $\ln (S/K_I)$ when $18.2\mu M=K_I\ll S\ll K_A=3mM$,  we consider $S_0=4K_I$ and choose the space domain
such that $5K_I<S(x)\leq K_A/5$. Therefore,  the computational domain depends on $G$.

Let $T, L$ be the characteristic time and space scale for the movement on the population level. Let $T_M, L_M$ be the characteristic time and length for the outside signal. We nondimensionalize equation \eqref{eq:q} by letting 
\begin{align*}
     t = T \tilde{t} \,,
\qquad 
     x = L \tilde{x} \,,
\qquad 
     v_0 = V_0 \tilde{v} \,,
\qquad 
     k_R = \frac{1}{T_a} \tilde k_R \,,
\qquad
     Z(a) = \frac{1}{T_t} \tilde Z(a) \,,
\end{align*}
where $T_a$ and $T_t$ are the characteristic adaptation time and tumbling time
respectively. Let 
\begin{align*}
   \tilde{q}(\tilde{t},\tilde{x},v,a)=q(t,x,v,a) \,,
\qquad
  \tilde M (\tilde t, \tilde x) = M(t, x) \,.
\end{align*}
Then the equation for $\tilde{q}$ becomes 
\begin{align} \label{eq:scale}
   \frac{L}{TV_0}\del_{\tilde{t}} \tilde{q} 
   + \tilde{v} \del_{\tilde{x}}\tilde{q} 
   + N \alpha_0 a(1 - a) \del_a \vpran{\vpran{-\frac{L}{V_0 T_M}\partial_{\tilde t}\tilde M-\frac{L}{L_M}\tilde v \partial_{\tilde x} \tilde M + \frac{L}{V_0 T_a}\tilde{k}_R \vpran{1 - \frac{a}{a_0}}} \tilde{q}}\nonumber
  \\ =\frac{L}{V_0 T_t} \tilde{Z}(a) \int_\VV (\tilde{q}(\tilde{t}, \tilde{x}, v', a) - \tilde{q}(\tilde{t}, \tilde{x}, v, a)) \dv' \,.
\end{align}
In the exponential environment, let
\begin{align*}
     T_M = \infty \,,
\qquad 
     V_0 =10 \mu m/s \,.
\end{align*}
The scalings for Case I and II in the previous section correspond to
\begin{itemize}
\item In case I where $g = \BigO(1)$,
let
\begin{align*}
    \Eps = \frac{V_0 T_a}{L} = \frac{V_0 T_t}{L} = \frac{L_M}{L}  \,.
\end{align*}
Then 
\begin{align*}
    g = \BigO\vpran{\frac{V_0 T_a}{L_M}} = \BigO(1) \,.
\end{align*}
Hence equation\eqref{eq:scale} becomes equation~\eqref{eq:q-1-scaled}. 
\item In case II, let
\begin{align*}
   \Eps = \frac{V_0 T_a}{L} = \frac{V_0 T_t}{L} \,,
\qquad
   \Eps^\mu = \frac{L}{T V_0} \,,
\qquad
   \Eps^{1-\mu} = \frac{L_M}{L} \,,
\end{align*}
%
Then 
\begin{align*}
    g = \BigO\vpran{\frac{V_0 T_a}{L_M}} 
    = \BigO\vpran{\frac{V_0 T_a}{L}} \BigO\vpran{\frac{L}{L_M}}
    = \Eps^\mu \,.
\end{align*}
Thus equation~\eqref{eq:scale} becomes equation~\eqref{eq:q-3-scaled}.
\end{itemize}

In \cite{SWOT}, the authors developed a macroscopic pathway-based mean field theory (PBMFT) which successfully explained a counter-intuitive experimental observation: there exists a phase shift between the dynamics of ligand concentration and centre of mass of the cells in a spatial-temporal fast-varying environment,. However, PBMFT fails to give the right macroscopic drift velocity in the exponential environment with  large gradients. 

In the rest of this section we compare our results with SPECS and PBMFT. Exponential environment is considered and we use 
periodic boundary conditions in space, i.e. in SPECS, each bacterial that runs out of the right (left) boundary of computational domain will enter again from the left (right) with the same activity $a$.

\begin{itemize}
\item  {\em Comparison of the distribution in $a$.} We compute the distribution of the bacteria in $a$ in two ways: one is to run SPECS and count the number of bacteria with $a$ in a small interval; the other is to
 compute $\frac{q_0^+}{N\alpha_0 a(1-a)}$, $\frac{q_0^-}{N\alpha_0 a(1-a)}$ analytically according to \eqref{soln:q-0-plus-1}, \eqref{soln:q-0-minus-1}.
 We can see that the analytical distribution yields almost the same distribution as SPECS. Moreover,
the average drift velocity of our macroscopic model matches well with SPECS, while the part that the PBMFT is no longer valid is in Case I and Case II with $0 < g <1$ where the hyperbolic scaling applies.

The formal asymptotic analysis in section \ref{sec:asymptotics} shows that the classification of the various cases depends on the size of $g = \frac{vG}{\alpha_0 k_R}$. Then for a given $k_R$,
we can divide the value of $G$ into several intervals, where each interval corresponds to one case. Fix $k_R=0.005s^{-1}$. Then
\begin{align*}
   g = 1 \Leftrightarrow G =7.4*10^{-4}\mu m^{-1} \,,
\qquad
   g = 0.1 \Leftrightarrow G =7.4*10^{-5}\mu m^{-1} \,,
\\
   g = 0.01 \Leftrightarrow G =7.4*10^{-6}\mu m^{-1} \,,
\qquad
  g < 0.01 \Leftrightarrow G < 7.4*10^{-6}\mu m^{-1} \,.
\end{align*}
Thus we can divide the range of $G$ as
\begin{itemize}
\item $G > 7.4*10^{-4}\mu m^{-1}$: Case I with $g > 1$. In this case the leading-order distribution $q_0$ spreads over $a \in (0, 1)$. The macroscopic density $\rho_0$ satisfies a hyperbolic equation.

\item $G \in(7.4*10^{-5}\mu m^{-1},7.4*10^{-4}\mu m^{-1})$: Case I with $0 < g < 1$. In this case the leading-order distribution $q_0$ is compactly supported on $[a_1, a_2]$. The macroscopic density $\rho_0$ satisfies a hyperbolic equation.

\item $G\in (7.4*10^{-6}\mu m^{-1},7.4*10^{-5}\mu m^{-1})$: Case II with $0 < \mu < 1$. In this case the leading-order distribution $q_0$ is concentrated at $a = 1/2$. The macroscopic density $\rho_0$ satisfies a hyperbolic equation.

\item $G<7.4*10^{-6}\mu m^{-1}$: Case II with $\mu = 1$. In this case the leading-order distribution $q_0$ is concentrated at $a = 1/2$. The macroscopic density $\rho_0$ satisfies the Keller-Segel equation. 
\end{itemize}
%
Figure \ref{fig:case12} and \ref{fig:case34} shows the analytical distributions given by the asymptotic analysis in Section~\ref{sec:asymptotics}
yield almost the same distributions as SPECS. As $G$ increases, more and more bacteria become concentrated near $a=0$. This indicates that the tumbling frequency of the bacteria becomes low. The density distribution is concentrated near $a = 0.5$ for $G$ small and it spreads out when $G$ increase. The moment closure techniques in 
all previous paper \cite{ErbanOther04,XO,STY,SWOT} have used the assumption that the methylation level is not far away from its average, so that it is possible to use 
the Taylor expansion near the average to approximate the distribution in the internal state. This assumption fails in the large-gradient environment. 


\item {\em The distribution of $\frac{q_0^+}{N\alpha_0 a(1-a)}$ and $\frac{q_0^-}{N\alpha_0 a(1-a)}$ near $a=0$.} According to the analytical formulas  in \eqref{soln:q-0-plus-1}-\eqref{soln:q-0-minus-1}, if $\theta_0=-\frac{z_0}{4N\alpha_0k_R}\frac{1}{a_1a_2}>1$, then $\frac{q_0^\pm}{N\alpha a(1-a)}\to 0$ as $a\to 0$. If $0<\theta_0<1$, then $\frac{q_0^\pm}{N\alpha a(1-a)}\to +\infty$ as $a\to 0$. This can be considered as a phase transition of the density distribution at $a=0$, which can be seen from Figure \ref{fig:disa0}. 
The different distributions of $q_0^+$ near $a=0$ for different cases are harder to distinguish from the SPECS simulation.

\item  {\em Comparison of the average drift velocity for different $k_R$'s and different $G$'s.} From Figure \ref{fig:case12} and Figure \ref{fig:case34}, we can observe that the distribution in $a$ is almost
uniform in space, while the fluctuation in space increases with $G$. We compute the average drift velocity analytically both by \eqref{eq:kappa1} and by SPECS. In the SPECS simulation, the population-level drift velocity is obtained by counting the difference between the number of forward and backward moving bacteria and multiply it by $v_0$.  In Figure \ref{fig:vd} we compare the average drift velocity obtained by these two methods as well as by PBMFT in \cite{SWOT}. The authors pointed out in \cite{SWOT} that the average drift velocity will saturate when $G$ increases. This is due to the particular stopping criteria that is used to determined when the system has arrived at a steady state. If instead we run the SPECS code for a longer time until the mean and variance of the average drift velocities do not change much, then the average drift velocities do not saturate but decrease when $G$ is large enough. As has already been observed in \cite{SWOT}, PBMFT can not give the right prediction of the population level chemotaxis velocity when $G$ becomes large while our analytical results match well with SPECS.

\end{itemize}

\section{Rigorous Derivation}
In this section, we rigorously derive the macroscopic models in all the cases in Section~\ref{sec:asymptotics}. 

\subsection{Well-posedness} The well-posedness of the kinetic equation~\eqref{eq:kinetic-m} will be established in the space of probability measures. To this end, we introduce a few notations from mass transportation. The space we will consider is $\CalP_1(\X)$, the probability space on the metric space $\X$ with finite first moments. In this paper, the metric space $\X$ is $\X = \R \times \VV \times (0, 1)$ where $\R$ and $(0, 1)$ are equipped with the usual Euclidean metric and $\VV$ is a bounded space with a unit measure $\dv$.
We use the 1-Wasserstein distance on $\CalP_1(\X)$ defined by
\begin{align*}
    W_1(\mu, \nu) 
    = \sup \left\{\int_{\X} \phi(x) \vpran{\dmu - \dnu} \Big| 
                       \norm{\phi}_{Lip} \leq 1\right\} \,,
\qquad
   \mu, \nu \in \CalP_1 \,.
\end{align*}

Let $E$ be a vector field and $X$ be transported by $E$ as
\begin{align*}
   \frac{{\rm d} X}{\dt} = E(t, X) \,,
\qquad
   X(0, x_0) = x_0 \,.
\end{align*}
Denote the associated flow map as $\CalT$ such that $\CalT x_0 = X$.
The push-forward operator $\CalT^t_E \# f_0$ is defined as
\begin{align*}
   \int_{\X} \xi(x) \vpran{\CalT^t_E \# f_0} (t, x)\dx
 = \int_{\X} \xi(\CalT x_0) f_0(x_0) \dx_0 \,,
\end{align*}
for any $\xi \in C_b(\X)$ where $C_b(\X)$ is the space of continuous and bounded functions on $\X$. For regular enough $E$ and $f_0$, the push-forward operator gives the solution to the transport equation
\begin{align*}
    \del_t f + \nabla_x \cdot \vpran{E(t, x) f} = 0 \,,
\qquad
   f(0, x) = f_0(x) \,.
\end{align*}
Define
\begin{align*}
   \hat q = \frac{q}{N\alpha_0 a(1-a)} \,.
\end{align*}
Then the equation for $\hat q$ becomes
\begin{align} \label{eq:hat-q}
   \del_t \hat q + v \del_x \hat q 
   + \del_a \vpran{\vpran{- v G + k_R \vpran{1 - \frac{a}{a_0}}} N \alpha_0 a(1 - a) \hat q}
   = Z(a) \int_\VV (\hat q(t, x, v', a) - \hat q(t, x, v, a)) \dv' \,.
\end{align}
The characteristic equation associated with equation~\eqref{eq:hat-q} is 
\begin{align*}
   \frac{\dx}{\dt} &= v \,,
\\
   \frac{\da}{\dt} &= \vpran{- v G + k_R \vpran{1 - \frac{a}{a_0}}} N \alpha_0 a(1 - a) \,,
\\
   \frac{\dv}{\dt} &= 0 \,.
\end{align*}
Thus the vector field $E$ is 
\begin{align} \label{def:E}
   E(t, x, a, v) 
   = \vpran{v, \,\, \vpran{- v G + k_R \vpran{1 - \frac{a}{a_0}}} N \alpha_0 a(1 - a), \,\, 0} \,,
\end{align}
which is globally Lipschitz for each given $G$. 

\begin{definition}
The measure-value solution to equation~\eqref{eq:q-1} is defined as the measure 
\begin{align*}
   q = N\alpha_0 a(1-a) \hat q \,,
\end{align*} 
where $\hat q \in C([0, T); \CalP_1\vpran{\R \times \VV \times (0,1))}$ satisfies
\begin{align} \label{def:soln-hat-q}
   \hat q(t,x, a, v) 
   = \CalT^t_E \# \hat q^{in} 
      + \int_0^t  \CalT^{t-s}_E \# \CalL \hat q (s, x, a, v) \ds \,.
\end{align}
Here $E$ is given in~\eqref{def:E} and
\begin{align*}
    \hat q^{in} = \frac{q^{in}}{N\alpha_0 a(1-a)} \in \CalP_1 \,,
\qquad
   \CalL \hat q = \vint{\hat q}_v - \hat q \,.
\end{align*}
\end{definition}

We recall one lemma from \cite{CCR2011}:
\begin{lem}[Lemma 3.18 in \cite{CCR2011}]\label{lem:CCR}
Let $\CalT: \X \to \X$ be a globally Lipschitz map and $f, g \in \CalP_1(\X)$. Then
\begin{align*}
   W_1(\CalT \# f, \CalT \# g)
\leq 
   Lip(\CalT) W_1(f, g) \,,
\end{align*}
where $Lip(\CalT)$ is the Lipschitz constant of $\CalT$.
\end{lem}

Applying Lemma~\ref{lem:CCR} to $\CalT = \CalT^t_E$ gives
\begin{lem} \label{lem:bound-W-1}
Let $\hat q^{in}_1, \hat q^{in}_2 \in \CalP_1(\X)$ and $\CalT^t_E$ be the flow map with vector field $E$ given in~\eqref{def:E}. Then
\begin{align*}
   W_1 \vpran{\CalT^t_E \# \hat q^{in}_1, \CalT^t_E \# \hat q^{in}_2}
\leq
   e^{t Lip(E)} W_1\vpran{\hat q^{in}_1, \hat q^{in}_1} \,.
\end{align*}
\end{lem}
\begin{proof}
By Lemma~\ref{lem:CCR}, we only need to estimate the Lipschitz bound of $\CalT^t_E$. Let $X= (x, a, v)$. Then
\begin{align*}
   \frac{{\rm d}X}{\dt} = E(X) \,,
\qquad
  X \big|_{t=0} = x \,.
\end{align*}
Therefore, for any given initial states $x, y$, we have
\begin{align*}
    \frac{\rm d}{\dt} \abs{T^t_E(x) - T^t_E(y)}
   = \frac{\rm d}{\dt} \abs{X - Y}
\leq
    \abs{E(X) - E(Y)}
\leq 
    Lip(E) \abs{X - Y} \,.
\end{align*}
By Gronwall's inequality, we have
\begin{align*}
    \abs{T^t_E(x) - T^t_E(y)}
\leq 
   e^{t Lip(E)} \abs{x-y} \,.
\end{align*}

\end{proof}

The main well-posedness result states
\begin{thm}
 \label{thm:well-posedness}
Suppose the intracellular dynamics and tumbling frequency are defined by~\eqref{assump:f-Lambda} with $G$ given. Suppose the initial data $q^{in}$ satisfies
\begin{align*}
    \frac{q^{in}}{N\alpha_0 a(1-a)} \in \CalP_1 \vpran{\R \times \VV \times (0, 1)} \,.
\end{align*}
 Then 
\begin{itemize}
\item[(a)]
for any $T > 0$, equation~\eqref{eq:hat-q} has a unique solution $\hat q \in C([0, T); \CalP_1(\R \times \VV \times \R^+))$ in the sense of~\eqref{def:soln-hat-q}. 

\item[(b)]
Let $S(t)$ be the solution operator to~equation~\eqref{eq:hat-q}. Then the equation is stable in the sense that
\begin{align} \label{bound:stability}
   W_1(S(t)\hat q^{in}_1, S(t) \hat q^{in}_2 )
\leq  e^{(Lip(E) + 2) t} W_1(\hat q^{in}_1, \hat q^{in}_2) \,,
\end{align}
where $Lip(E)$ is the global Lipschitz constant of $E$.


\end{itemize}
\end{thm}
\begin{proof}
For the ease of notation, in this proof we always denote
\begin{align*}
    \X = \R \times \VV \times (0, 1) \,.
\end{align*}
The well-posedness of~\eqref{def:soln-hat-q} will be shown by a fixed-point argument. We comment that if the initial data is smooth enough, such as $\hat q^{in} \in L^\infty(\X)$ with a bounded second moment, then the well-posedness has been established in the literature \cite{PTV}. In this case one has the maximum principle such that if $\hat q^{in} \geq 0$, then $\hat q \geq 0$ for all $t \in [0, T)$.

For the measure-valued case, denote the operator $\Gamma$ as 
\begin{align*}
   \Gamma \hat q
 = \CalT^t_E \# \hat q^{in} 
      + \int_0^t  \CalT^{t-s}_E \# \CalL \hat q (s, x, a, v) \ds \,.
\end{align*}
Define the space $C([0, T); \CalP_1^+(\X))$ as
\begin{align*}
& \quad \,
    C([0, T); \CalP_1^+(\X)) 
\\
& = \left\{\mu \in C([0, T); \CalP_1(\X)) \Big| \Gamma^k \mu \,\, \text{is a nonnegative measure for any $t \in [0, T)$ and any $k \geq 1$} \right\} \,.
\end{align*}
Note that $C([0, T); \CalP_1^+(\X))$ is non-empty since it contains all the regular solutions with $\hat q^{in} \geq 0$. The metric on $C([0, X); \CalP_1^+(\X))$ is 
\begin{align*}
    d\vpran{\hat q_1, \hat q_2}
 = \sup_{[0, T)} W_1 \vpran{\hat q_1(t, \cdot), \hat q_2(t, \cdot)} \,.
\end{align*}

We will show that $\Gamma$ is a contraction mapping on a convex subset of $C([0, X); \CalP_1^+(\X))$ for $T$ small enough. First, given $\hat q \in C([0, T); \CalP_1^+(\X))$, we verify that $\Gamma \hat q \in C([0, T); \CalP_1^+(\X))$. Indeed, for each fixed $t$, $\Gamma \hat q$ is a nonnegative measure with its two parts satisfying  
\begin{align*}
   \int_\R \int_\VV \int_0^1 \CalT^t_E \# \hat q^{in} (x, a, v) \da \dv \dx = 1\,,
\end{align*}
and
\begin{align*}
   \int_\R \int_\VV \int_0^1 \int_0^t \CalT^{t-s}_E \# \vint{\hat q}_v (x, a) \da \dv \dx 
   = \int_\R \int_\VV \int_0^1 \int_0^t \CalT^{t-s}_E \# \hat q (x, a) \da \dv \dx\,.
\end{align*}
This shows $\Gamma \hat q(t, \cdot)$ is a probability measure for each $t \in [0, T)$.
Furthermore, the first moment of $\Gamma \hat q$ satisfies
\begin{align*}
& \quad \,
   \int_\R \int_\VV \int_0^1 |x| \Gamma \hat q(t, x, a, v) \da \dv \dx
\\
&\leq 
   \int_\R \int_\VV \int_0^1 |x| \CalT^t_E \# \hat q^{in} (x, a, v) \da \dv \dx
   + \int_0^t \int_\R \int_\VV \int_0^1 |x| \CalT^t_E \# \vint{\hat q}_v (x, v, a) \da \dv \dx
\\
& \leq
   \int_\R \int_\VV \int_0^1 (v_0 T + |x|) \hat q^{in} (x, a, v) \da \dv \dx
    + \int_0^T \int_\R \int_\VV \int_0^1 (v_0 T + |x|) \hat q (x, a, v) \da \dv \dx
\\
&< \infty \,,
\end{align*}
where $\Disp v_0 = \max_{\VV}{|v|}$. Therefore $\Gamma \hat q(t, \cdot) \in C([0, X); \CalP_1^+(\X))$. 

Next, let $\hat q_1, \hat q_2 \in C([0, X); \CalP_1^+(\X))$ with $\hat q_1^{in} = \hat q_2^{in}$. Then
\begin{align*}
   \sup_{[0, T)} W_1\vpran{\Gamma \hat q_1, \Gamma \hat q_1}
&\leq 
  2 T \sup_{[0, T)} W_1 \vpran{\CalT^{t}_E \# \hat q_1, \,,
     \CalT^{t}_E \# \hat q_2 (s, x, a, v)}
\\
& \leq 
   2 T e^{T Lip(E)} \sup_{[0, T)}W_1\vpran{\hat q_1, \hat q_2} \,,
\end{align*}
where $Lip(E)$ is the global Lipschitz constant of $E$. Hence if $T$ is small enough,  then $\Gamma$ is a contraction mapping when restricted to the convex subset of $C([0, X); \CalP_1^+(\X))$ with all the elements having the same initial data. Repeating the proof of the contraction mapping theorem, one can find a fixed point $\hat q$ in $C([0, X); \CalP_1^+(\X))$, which satisfies that $\hat q(0, \cdot) = \hat q^{in}$. Therefore for any initial data in $C([0, X); \CalP_1^+(\X))$, equation~\eqref{eq:hat-q} has a unique solution. Moreover, this solution can be extended to and $T > \infty$ since the bound of $T$ is independent of the solution. 

The restriction of the initial data will be removed after we prove the stability of the equation. The stability stated in~\eqref{bound:stability} is shown as follows. Let $\hat q_1, \hat q_2$ be two solutions in $C([0, X); \CalP_1^+(\X))$. Then for each $t \in [0, T)$, we have
\begin{align*}
   W_1(\hat q_1, \hat q_2)
&\leq
   W_1 \vpran{\CalT^t_E \# \hat q^{in}_1, \CalT^t_E \# \hat q^{in}_2}
   + 2\int_0^T W_1 \vpran{\CalT^{t-s}_E \# q_1, \CalT^{t-s}_E \# q_1}
\\
& \leq 
    e^{t Lip(E)} W_1 \vpran{\hat q^{in}_1,\hat q^{in}_2}
    + 2 \int_0^t e^{(t-s) Lip(E)}W_1 \vpran{q_1, q_1}(s) \ds \,.
\end{align*}
Hence, 
\begin{align*}
   e^{-t Lip(E)}W_1(\hat q_1, \hat q_2)
& \leq 
    W_1 \vpran{\hat q^{in}_1,\hat q^{in}_2}
    + 2 \int_0^t e^{-s Lip(E)}W_1 \vpran{q_1, q_1}(s) \ds \,.
\end{align*}
By Gronwall's inequality, we obtain that
\begin{align*}
   W_1(\hat q_1, \hat q_2)
\leq 
  e^{(Lip(E) + 2) t} W_1\vpran{\hat q^{in}_1,\hat q^{in}_2} \,.
\end{align*}
Since the initial data of measures in $C([0, T); \CalP_1^+(\X))$ is dense in $\CalP_1(\X)$, by a density argument and the stability result, equation~\eqref{eq:hat-q} has a unique solution in $C([0, T); \CalP_1^+(\X))$ for any initial data in $\CalP_1(\X)$. 
\end{proof}

\subsection{Asymptotics}  The main idea in proving the convergence is to show that $\vint{q_\Eps}_x$ has no accumulation at the boundary $a=0, 1$, that is, to show that the family of the probability measures $\vint{\hat q_\Eps}_x$ is tight. We will impose extra conditions (in addition to those for the well-posedness) to the initial data for each of the cases. 
%
%
%
%
\subsubsection{\bf Case I: $g = \BigO(1)$}  Recall the scaled equation
\begin{align*} 
   \Eps \del_t q_\Eps + \Eps v \del_x q_\Eps 
   + N \alpha_0 a(1 - a) \del_a \vpran{(-vG + k_R (1 - 2a)) q_\Eps}
   = Z(a) \CalL q_\Eps \,.
\end{align*}
More specifically, the discrete model has the form
\begin{align} 
   \Eps \del_t q_\Eps^+ + \Eps v_0 \del_x q_\Eps^+ 
   - 2 k_R N \alpha_0 a(1 - a) \del_a \vpran{(a - a_1) q_\Eps^+}
   = \frac{Z(a)}{2} (q_\Eps^- - q_\Eps^+)  \,,
   \label{eq:q-1-scaled-plus}
\\
   \Eps \del_t q_\Eps^- - \Eps v_0 \del_x q_\Eps^- 
   + 2 k_R N \alpha_0 a(1 - a) \del_a \vpran{(a_2 - a) q_\Eps^-}
   = \frac{Z(a)}{2} (q_\Eps^+ - q_\Eps^-)  \,,
   \label{eq:q-1-scaled-minus}
\end{align}
where $a_1, a_2$ are defined in~\eqref{def:g-a-1-2}.

We will separate the two cases for $g > 1$ and $0 < g < 1$. First, if $g > 1$, then $a_1 < 0 < a < 1 < a_2$.
\begin{prop} \label{prop:case-1}
Let $Q_0(v, a)$ (or $Q_0^\pm(a)$) be defined in~\eqref{soln:q-0-plus-1} and~\eqref{soln:q-0-minus-1}. Suppose in addition to the assumptions in Theorem~\ref{thm:well-posedness} that the initial condition satisfies
\begin{align} \label{cond:bound-initial-1}
     \vint{q_\Eps(0, \cdot, v, a)}_x \leq \beta_0 \, Q_0(v, a)
\end{align}
for some $\beta_0 > 1$. Then 
\begin{itemize}
\item[(a)] $\vint{q_\Eps(t, \cdot, v, a)}_x \leq \beta_0 \, Q_0(v, a)$ for all $t \geq 0$.  \medskip
\item[(b)] $\vint{q_\Eps}_{x} \to Q_0(v, a)$ as measures. 
\end{itemize}
\end{prop}
\begin{proof} 
(a) The measure $\vint{q_\Eps}_x$ satisfies the equation (in the sense of distributions)
\begin{align*} 
   \Eps \del_t \vint{q_\Eps}_x  
   + N \alpha_0 a(1 - a) \del_a \vpran{(-vG + k_R (1 - 2a)) \vint{q_\Eps}_x}
   = Z(a) \CalL \vint{q_\Eps}_x \,.
\end{align*}
Fix $\Eps > 0$. We first use the same argument for proving the maximum principle for transport equations for the initial data $q_\Eps \in L^1(\R \times \VV \times (0, 1))$. To this end, let
\begin{align*}
   (\vint{q_\Eps}_x - \beta_0 Q_0)^+
   = \begin{cases}
       \vint{q_\Eps}_x -\beta_0 Q_0 \,, & \text{if $\vint{q_\Eps}_x -\beta_0 Q_0 > 0$,} \\[2pt]
       0 \,, & \text{otherwise} \,.
       \end{cases}
\end{align*}
Then $(\vint{q_\Eps}_x - \beta_0 Q_0)^+$ satisfies (in the sense of distributions)
\begin{align} 
& \quad \,
   \Eps \del_t (\vint{q_\Eps}_x - \beta_0 Q_0)^+ 
   + N \alpha_0 a(1 - a) \del_a \vpran{(-vG + k_R (1 - 2a)) (\vint{q_\Eps}_x - \beta_0 Q_0)^+} \nn
\\
&   = Z(a) \vpran{\text{sgn}^+(\vint{q_\Eps}_x - \beta_0 Q_0)} \CalL \vpran{\vint{q_\Eps}_x - \beta_0 Q_0} \,, \label{eq:q-beta-p-0-plus}
\end{align}
where the positive sign function is
\begin{align*}
   \text{sgn}^+(\vint{q_\Eps}_x - \beta_0 Q_0)
   = \begin{cases}
       1 \,, & \text{if $\vint{q_\Eps}_x - \beta_0 Q_0 > 0$,} \\[2pt]
       0 \,, & \text{otherwise} \,.
       \end{cases}
\end{align*}
By the definition of $\CalL$, we have
\begin{align*}
   \vint{\vpran{\text{sgn}^+ g} \CalL g}_v
= \vpran{\text{sgn}^+ g} \int_\VV g(v) \dv - \int_\VV g^+ \dv
\leq 
   \int_\VV g^+ \dv - \int_\VV g^+ dv = 0 \,.
\end{align*}
Therefore, if we integrate~\eqref{eq:q-beta-p-0-plus} with respect to $x, v, a$ with weight $\Disp\frac{1}{a(1-a)}$, then
\begin{align*}
   \Eps \del_t \int_\R \int_\VV \int_0^1
       \frac{(\vint{q_\Eps}_x - \beta_0 Q_0)^+}{a(1-a)} \da \dv \dx
   \leq 0 \,.
\end{align*}
Since initially $(\vint{q_\Eps(0, \cdot, v, a)}_x - \beta_0 Q_0)^+ = 0$, we have $(\vint{q_\Eps}_x - \beta_0 Q_0)^+ = 0$ for all $t \geq 0$, which proves the upper bound in part (a) if the initial data is in $L^1(\R \times \VV \times (0, 1))$. For each fixed $\Eps$, the stability result~\eqref{bound:stability} applies (with $Lip(E)$ changed to $\frac{1}{\Eps}Lip(E)$). Thus we can extend to the general case by a density argument. 

\medskip\smallskip

\Ni (b) The bound in part (a) implies that the family of probability measures $\left\{\vint{\frac{q_\Eps}{N\alpha_0 a (1-a)}}_{x} \right\}$ is tight (in $v, a$) since $Q_0$ decays algebraically at $a = 0, 1$. Thus there exists a subsequence $\vint{q_{\Eps_k}}_{x}$ and a probability measure $\tilde Q_0(v, a)$ such that $\vint{\frac{q_{\Eps_k}}{N\alpha_0 a (1-a)}}_{x} \to \tilde Q_0$ as measures. This also gives
\begin{align*}
    \vint{q_{\Eps_k}}_v \to N \alpha_0 a(1-a) \tilde Q_0 
\qquad
   \text{as measures.}
\end{align*}
Since equations~\eqref{eq:q-1-scaled-plus}-\eqref{eq:q-1-scaled-minus} are linear, the limit $N \alpha_0 a(1-a)\tilde Q_0$ must satisfy the ODE~\eqref{eq:q-0-1}. By the condition that $\tilde Q_0$ is a probability measure and the uniqueness of solutions to \eqref{eq:q-0-1} with the normalization, we conclude that the full sequence $\vint{q_\Eps}_x \to Q_0$ as measures. 
\end{proof}
\begin{rmk}
Bound \eqref{cond:bound-initial-1} is the same as assuming $q_\Eps(0, \cdot, \cdot, \cdot) \in L^\infty(\R \times \VV \times (0, 1))$ and has at least the same algebraic decay rate as the leading order $Q_0$ at $a=0, 1$. 
\end{rmk}

Next we consider the case where $0 < g < 1$. In this case we have $0 < a_1 < 1/2 < a_2 < 1$. 
\begin{prop} \label{prop:case-2}
Let $Q_0(v, a)$ (or $Q_0^\pm(a)$) be defined in~\eqref{soln:q-0-plus-1}-\eqref{soln:q-0-minus-1} with the condition that $\frac{v_0}{k_R}G < 1$.
Let $\phi \in C^1(0, 1)$ be a convex function which satisfies that 
\begin{itemize}
\item
$\phi = 0$ on $[a_1, a_2]$,
\item
$\phi$ is decreasing on $(0, a_1]$ and $\phi(a) \to \infty$ as $a \to 0$,
\item
$\phi$ is increasing on $[a_2, 1)$ and $\phi(a) \to \infty$ as $a \to 1$.
\end{itemize}
Suppose in addition to the assumptions in Theorem~\ref{thm:well-posedness}, the initial data $q_\Eps(0, x, v, \cdot)$ satisfies the bound
\begin{align} \label{cond:initial}
   \vint{\phi(a) \frac{q_\Eps(0, x, v, \cdot)}{N\alpha_0 a (1-a)}}_{x, v, a} < \beta_1 < \infty \,,
\end{align}   
where the constant $\beta_1 > 0$ is independent of $\Eps$. Then 

\medskip
\Ni (a) for all $t \geq 0$, it holds that
\begin{align*}
    \vint{\phi(a) \frac{q_\Eps(t, x, v, \cdot)}{N\alpha_0 a (1-a)}}_{x, v, a} < \beta_1 \,.
\end{align*}

%

\medskip
\noindent (b) $\vint{q_\Eps}_{x} \to Q_0(v, a)$ as measures. 
\end{prop}

\begin{proof}
(a) Similar as in Proposition~\ref{prop:case-1}, we only need to consider the initial data in $L^1(\X \times \VV \times (0, 1))$ and then apply the density argument for each fixed $\Eps$. Multiply \eqref{eq:q-1-scaled-plus} and \eqref{eq:q-1-scaled-minus} by $\phi$ and integrate in $x, a$. Then
\begin{align*}
& \quad \,
  \Eps \del_t \int_\R \int_0^1 \phi(a) \vpran{q_\Eps^+ (t, x, a) + q_\Eps^-(t, x, a)} \frac{1}{a(1-a)} \dx\da
\\
&  = - 2k_R N \alpha_0 \int_\R \int_0^1 \phi'(a) (a - a_1) q_\Eps^+ (a) \dx\da
     + 2 k_R N \alpha_0 \int_\R \int_0^1 \phi'(a) (a_2 - a) q_\Eps^- (a) \dx\da
\leq 0 \,. 
\end{align*}
Therefore we have
\begin{align*}
& \quad \,
  \int_\R \int_0^1 \phi(a) \vpran{q_\Eps^+ (t, x, a) + q_\Eps^-(t, x, a)} \frac{1}{a(1-a)} \dx\da
\leq  \vint{\phi q_\Eps(0, x, v, \cdot)}_{v, a} < \beta_1 \,.
\end{align*}

\Ni (b) The bound in part (a) again shows that the family of probability measures $\left\{\vint{\frac{q_\Eps}{N \alpha_0 a (1-a)}}_x \right\}$ is tight. Thus by a similar argument as in part (b) for Proposition~\ref{prop:case-1}, we have $\vint{q_\Eps}_{x} \to Q_0(v,a)$ as measures. 
\end{proof}

\subsubsection{\bf Case II: $g = \BigO\vpran{\Eps^\mu}$ with $0 < \mu \leq 1$} The scaled equation in this case is
\begin{align} \label{eq:q-2-scaled-recall}
   \Eps^{1+\mu} \del_t q_\Eps + \Eps v \del_x q_\Eps 
   + N \alpha_0 a(1-a) \del_a \vpran{\vpran{-v \Eps^\mu G_\mu + k_R (1 - 2a)} q_\Eps}
   = Z(a) \CalL{q_\Eps} \,.
\end{align}
The main result for Case II is 
\begin{prop}
Suppose $\Eps$ is small enough and $q_\Eps$ is a measure-valued solution to~\eqref{eq:q-2-scaled-recall}. Suppose the initial data $q_\Eps(0, x, v, a)$ satisfies the bound~\eqref{cond:initial}.
\begin{itemize}
\item[(a)] If $0 < \mu < 1$, then we have $q_\Eps \to q_0$ as measures where $q_0 = \rho_0(t, x) \delta_{1/2}(a)$ and $\rho_0$ satisfies the transport equation~\eqref{eq:macro-2}.

\item[(b)] If $\mu = 1$, then we have $q_\Eps \to q_0$ as measures where $q_0 = \rho_0(t, x) \delta_{1/2}(a)$ and $\rho_0$ satisfies the Keller-Segel equation~\eqref{eq:KS}.

\end{itemize}
\end{prop}
\begin{proof}
The convergence of $q_\Eps$ follows from a similar proof as for Case I with $0 < g < 1$ since we also have in Case II the condition that $0 < a_1 < 1/2 < a_2 < 1$. 

\medskip
\Ni (a) In order to show that $\rho_0$ satisfies equation~\eqref{eq:macro-2}, we first show some uniform-in-$\Eps$ estimate for $q_\Eps$.  Let $\delta >0$ be arbitrary. Let 
\begin{align*}
   \eta(a) = \frac{(1-2a)^2}{\sqrt{\delta + (1-2a)^2}} \,.
\end{align*}
Multiply $\eta(a)$ to equation~\eqref{eq:q-2-scaled-recall}, integrate in $x, a$, and add the two equations. This gives
\begin{align*}
& \quad \, 
   \Eps^{1+\mu} \del_t \int_\R \int_0^1
       \frac{\eta'(a) \vpran{q^+_\Eps + q^-_\Eps} }{N \alpha_0 a (1-a)} \da \dx
  - k_R \int_\R \int_0^1
        \eta'(a) (1-2a) \vpran{q^+_\Eps + q^-_\Eps}   \da\dx 
\\
&  = - \Eps^\mu v_0 G_\mu \int_\R 
         \int_0^1 \eta'(a) \vpran{q^+_\Eps - q^-_\Eps} \da \dx\,.
\end{align*}
where 
\begin{align*}
   \eta'(a) = -2 \frac{(1-2a)^3 + 2(1-2a) \delta}{\vpran{\sqrt{\delta + (1-2a)^2}}^3} \,.
\end{align*}
Therefore, by integrating in time we have
\begin{align} \label{bound:a-priori}
& \quad \,
   - k_R \int_0^T\int_\R \int_0^1
        \frac{\eta'(a) (1-2a)}{\Eps^\mu} \vpran{q^+_\Eps + q^-_\Eps}   \da\dx\dt \nn
\\
& = 
    2 k_R \int_0^T\int_\R \int_0^1
        \frac{(1-2a)^4 + 2 (1-2a)^2 \delta}{\vpran{\sqrt{\delta + (1-2a)^2}}^3}  \,\frac{q^+_\Eps + q^-_\Eps}{\Eps^\mu}  \da\dx\dt \nn
\\
& \leq  \Eps \int_\R \int_0^1
       \frac{|1- 2a| \vpran{q^+_\Eps(0, x, a) + q^-_\Eps(0, x, a)} }{N \alpha_0 a (1-a)} \da \dx 
\\
& \hspace{0.3cm} + 2v_0 G_\mu \int_0^T \int_\R 
         \int_0^1 \frac{(1-2a)^3 + 2(1-2a) \delta}{\vpran{\sqrt{\delta + (1-2a)^2}}^3} \vpran{q^+_\Eps - q^-_\Eps} \da \dx \dt \,, \nn
\end{align}
for any $\delta>0$. Note that for each $\delta>0$, we have
\begin{align*}
   \int_0^1 \frac{(1-2a)^3 + 2(1-2a) \delta}{\vpran{\sqrt{\delta + (1-2a)^2}}^3} \dsigma
= \int_{(0, 1/2)} \frac{(1-2a)^3 + 2(1-2a) \delta}{\vpran{\sqrt{\delta + (1-2a)^2}}^3} \dsigma
   + \int_{(1/2,1)} \frac{(1-2a)^3 + 2(1-2a) \delta}{\vpran{\sqrt{\delta + (1-2a)^2}}^3} \dsigma \,,
\end{align*}
for any probability measure $\sigma$. Let $\delta \to 0$. Then we have
\begin{align*}
& \quad \,
    2 k_R \int_0^T\int_\R \int_0^1
        \frac{|1-2a|}{\Eps^\mu}  \,\vpran{q^+_\Eps + q^-_\Eps}  \da\dx\dt \nn
\\
& \leq  \Eps \int_\R \int_0^1
       \frac{|1- 2a| \vpran{q^+_\Eps(0, x, a) + q^-_\Eps(0, x, a)} }{N \alpha_0 a (1-a)} \da \dx 
   + 2v_0 G_\mu \int_0^T \int_\R \int_{(0, 1/2)} (-1)  \vpran{q^+_\Eps - q^-_\Eps} \da \dx \dt
\\
& \hspace{0.3cm}
   + 2v_0 G_\mu \int_0^T \int_\R \int_{(1/2, 1)}  \vpran{q^+_\Eps - q^-_\Eps} \da \dx \dt \,.
\end{align*}
As a consequence, if we let $\Eps \to 0$, then
\begin{align} \label{limit:Case-3}
    2 k_R \int_0^T\int_\R \int_0^1
        \frac{|1-2a|}{\Eps^\mu}  \,\vpran{q^+_\Eps + q^-_\Eps}  \da\dx\dt
    \to 0
\quad
   \text{as $\Eps \to 0$.}   
\end{align}
Hence we have the limit
\begin{align*}
    q^+_\Eps + q^-_\Eps \to \rho_0(t, x) \delta_{1/2}(a) 
\qquad
   \text{as $\Eps \to 0$,}
\end{align*}
where $\rho_0$ is a probability measure. 
Let $\phi_1(t,x) \in C^\infty_c((0, T) \times \R)$. 
Multiply $\phi_1 v/Z(a)$ to equation~\eqref{eq:q-2-scaled-recall} and integrate in $(t, x, v, a)$. This gives
\begin{align} \label{eq:first-moment-3}
&   \Eps^{1+\mu} \vpran{\int_\R \int_0^1 
   \frac{\phi_1(t, x) v_0 \vpran{q^+_\Eps - q^-_\Eps}}
   {Z(a) N \alpha_0 a (1-a)} \da \dx \dt
   - 
   \int_\R \int_0^1 
   \frac{\phi_1(0, x) v_0 \vpran{q^+_\Eps(0, x, a) - q^-_\Eps(0, x, a)}}
   {Z(a) N \alpha_0 a (1-a)} \da \dx} \nn
\\
&  - \Eps^{1+\mu} \int_0^t \int_\R \int_0^1 
   \frac{\del_t\phi_1(\tau, x) v_0 \vpran{q^+_\Eps - q^-_\Eps}}
   {Z(a) N \alpha_0 a (1-a)} \da \dx \dtau
  - \Eps \int_0^t \int_\R \int_0^1
      \frac{\del_x\phi_1(\tau, x) v_0^2 \vpran{q^+_\Eps + q^-_\Eps}}
   {Z(a) N \alpha_0 a (1-a)} \da \dx \dtau \nn
\\ 
&  + \Eps^\mu v_0^2 G_\mu \int_0^t \int_\R \int_0^1 
       \phi_1(\tau,x) \vpran{\frac{1}{Z(a)}}' 
       \vpran{q^+_\Eps + q^-_\Eps} \da\dx\dtau 
\\
&  - k_R v_0 \int_0^t \int_\R \int_0^1
      \phi_1(\tau,x) \vpran{\frac{1}{Z(a)}}' (1 - 2a) \vpran{q^+_\Eps - q^-_\Eps} \da\dx\dtau \nn
\\
& = \int_0^t \int_\R \int_0^1 
       \frac{\phi_1(\tau, x) v_0 \vpran{q^+_\Eps - q^-_\Eps}}
       {N \alpha_0 a (1-a)}  \da\dx\dtau \,. \nn
\end{align}
Divide equation~\eqref{eq:first-moment-3} by $\Eps^\mu$ and pass $\Eps$ to zero. By the assumption that $0 < \mu < 1$, the first 4 terms on the left-hand side of~\eqref{eq:first-moment-3} vanish. The fifth term on the left satisfies the limit
\begin{align} \label{limit:1}
& \quad \,
   v_0^2 G_\mu \int_0^t \int_\R \int_0^1 
       \phi_1(\tau,x) \vpran{\frac{1}{Z(a)}}' 
       \vpran{q^+_\Eps + q^-_\Eps} \da\dx\dtau \nn
\\
&\to
   v_0^2 G_\mu \int_0^t \int_\R \int_0^1 
       \phi_1(\tau,x) \vpran{\frac{1}{Z(a)}}' 
       \rho_0(t, x) \delta_{1/2}(a) \da\dx\dtau
\\
& = \frac{N\alpha_0}{4}v_0^2 G_\mu \vpran{\frac{1}{Z(a)}}' \Big|_{a=1/2}  \int_0^t \int_\R  
        \phi_1(t, x) \rho_0(t, x) \dx\dt \,, \nn
\end{align}
By~\eqref{limit:Case-3}, the sixth term on the left-hand side of~\eqref{eq:first-moment-3} also vanishes. 
Summarizing all the limits, we get 
\begin{align} \label{limit:2}
   \frac{1}{\Eps^\mu} 
     \int_0^t \!\! \int_\R \int_0^1 
       \frac{\phi_1(t, x) v_0 \vpran{q^+_\Eps - q^-_\Eps}}
       {N \alpha_0 a (1-a)}  \da\dx\dt 
\to 
  \frac{N\alpha_0}{4}v_0^2 G_\mu \vpran{\frac{1}{Z(a)}}' \Big|_{a=1/2}  \int_0^t \int_\R  
        \phi_1(t, x) \rho_0(t, x) \dx\dt \,.
\end{align}
The weak formulation for $q_\Eps$ is
\begin{align} \label{eq:zero-moment-3}
&  \int_\R \int_0^1 
   \frac{\phi_3(t, x) \vpran{q^+_\Eps + q^-_\Eps}}
   {N \alpha_0 a (1-a)} \da \dx \dt
   - \int_\R \int_0^1 
   \frac{\phi_3(0, x) \vpran{q^+_\Eps(0, x, a) + q^-_\Eps(0, x, a)}}
   {N \alpha_0 a (1-a)} \da \dx \dt \nn
\\
&  - \int_0^t \int_\R \int_0^1 
   \frac{\del_t\phi_3(t, x) \vpran{q^+_\Eps + q^-_\Eps}}
   {N \alpha_0 a (1-a)} \da \dx \dt
  - \frac{1}{\Eps^\mu} \int_0^t \int_\R \int_0^1
      \frac{\del_x\phi_3(t, x) v_0 \vpran{q^+_\Eps - q^-_\Eps}}
   {N \alpha_0 a (1-a)} \da \dx \dt  
\\
&   = 0 \,. \nn
\end{align}
for $\phi_3 \in C^\infty_c(\R^+ \times \R)$. Take $\phi_1 = \del_x \phi_3$ in~\eqref{limit:2} pass to the limit in~\eqref{eq:zero-moment-3} then gives rise to the weak formulation of~\eqref{eq:macro-2} which reads
\begin{align*}
&  \int_\R \int_0^1 
   \frac{\phi_3(t, x) \vpran{q^+_\Eps + q^-_\Eps}}
   {N \alpha_0 a (1-a)} \da \dx 
   - \int_\R \int_0^1 
   \frac{\phi_3(0, x) \vpran{q^+_\Eps(0, x, a) + q^-_\Eps(0, x, a)}}
   {N \alpha_0 a (1-a)} \da \dx \nn
\\
&  - \int_0^t \int_\R \int_0^1 
   \frac{\del_t\phi_3(\tau, x) \vpran{q^+_\Eps + q^-_\Eps}}
   {N \alpha_0 a (1-a)} \da \dx \dtau
  -  \frac{N\alpha_0}{4}v_0^2 G_\mu \vpran{\frac{1}{Z(a)}}' \Big|_{a=1/2}  \int_0^t \int_\R  
        \del_x \phi_3(\tau, x) \rho_0(t, x) \dx\dtau
\\
&   = 0 \,.    
\end{align*}

\medskip
\Ni (b) The only difference in Case IV is that when $\mu = 1$, the fourth term on the left of~\eqref{eq:first-moment-3} survives and satisfies the limit
\begin{align} \label{limit:3}
  - \int_0^t \int_\R \int_0^1
      \frac{\del_x\phi_1(t, x) v_0^2 \vpran{q^+_\Eps + q^-_\Eps}}
   {Z(a) N \alpha_0 a (1-a)} \da \dx
\to 
  - \frac{v_0^2}{Z(1/2)} \int_0^t \int_\R
      \del_x\phi_1(t, x) \rho_0(t, x) \dx \dt
\end{align}
Therefore, in Case IV we have 
\begin{align} \label{limit:4}
& \quad \,
   \frac{1}{\Eps^\mu} 
     \int_0^t \int_\R \int_0^1 
       \frac{\phi_1(t, x) v_0^2 \vpran{q^+_\Eps - q^-_\Eps}}
       {N \alpha_0 a(1-a)}  \da\dx\dt  \nn
\\
&\to 
  - \frac{v_0^2}{Z(1/2)} \int_0^t \int_\R
      \del_x\phi_1(t, x) \rho_0(t, x) \dx \dt
  + \frac{N\alpha_0}{4}v_0^2 G_\mu \vpran{\frac{1}{Z(a)}}' \Big|_{a=1/2}  \int_0^t \int_\R  
        \phi_1(t, x) \rho_0(t, x) \dx\dt \,.
\end{align}
This gives~\eqref{eq:KS} in its weak formulation which reads
\begin{align*}
&  \int_\R \int_0^1 
   \frac{\phi_3(t, x) \vpran{q^+_\Eps + q^-_\Eps}}
   {N \alpha_0 a (1-a)} \da \dx 
   - \int_\R \int_0^1 
   \frac{\phi_3(0, x) \vpran{q^+_\Eps(0, x, a) + q^-_\Eps(0, x, a)}}
   {N \alpha_0 a (1-a)} \da \dx \nn
\\
&  - \int_0^t \int_\R \int_0^1 
   \frac{\del_t\phi_3(\tau, x) \vpran{q^+_\Eps + q^-_\Eps}}
   {N \alpha_0 a (1-a)} \da \dx \dtau
  -  \frac{N\alpha_0}{4}v_0^2 G_\mu \vpran{\frac{1}{Z(a)}}' \Big|_{a=1/2}  \int_0^t \int_\R  
        \del_x \phi_3(\tau, x) \rho_0(t, x) \dx\dtau
\\
&  + \frac{v_0^2}{Z(1/2)} \int_0^t \int_\R
      \del_{xx} \phi_3(t, x) \rho_0(t, x) \dx \dt  = 0 \,.
\end{align*}
We thereby finish the proof of the limits. 
\end{proof}

\section{Conclusion}
We derive advection and advection-diffusion macroscopic models for E.coli chemotaxis that match quantitatively with the agent-based model in the exponential environment with large gradients. The derivation is based on the parabolic or hyperbolic scalings of the kinetic-transport equation that couples the internal signal pathway. The scaling that we have
considered indicates that the time scale of the population level movement is longer than the individual bacteria adaptation and movement.

When $G$ is small, the drift velocity in Case II is proportional to $G$, which gives the logarithm sensing.
However, the  diffusion in the limiting macroscopic model of 
Case I is one order less than the advection, while the drift velocity does not linearly depend on $G$. This shows when $G$ becomes large, the logarithm sensing is no longer valid. Yet, we can give the drift velocity analytically thanks to the simple form of the adaptation rate.

It is well known that, when the Keller-Segel equation is coupled with an elliptic or parabolic equation for the chemical signal $S$ 
blowup may happen in finite time.
 Our results provide a possible mechanism to prevent the blowup phenomenon in the Keller-Segel model. It has been shown \cite{BCM,BDP} that if $\phi(\nabla S)=\nabla S$ and the initial mass goes beyond a critical level, then the Keller-Segel model exhibits nonphysical blowups in high dimensions. Various strategies have been proposed mathematically and biologically to prevent this nonphysical blow up \cite{CC,CW, HJ}. Some efforts have also been devoted to study the dynamics of the solution after the blowup in the sense of measures \cite{JV1,FV}. One of the biologically relevant assumptions is the "volume filling" effect, which takes into account that the bacteria do not want to jump to the place where the population is too crowded \cite{WH,CC}. However, this assumption is still phenomenologically.  We observe numerically that if we further increase the chemical gradient, the average drift velocity decreases to a constant. This suggests that one physical way to prevent the blowup in the Keller-Segel model is to choose the dependence of the advection on the signal gradient  $\phi(\nabla S)$ such that $\phi(u)$ increases with $u$ to a maximum value and then decreases to a constant as $\abs{u}$ further increases.  
    
The numerical results in this paper show that our analytical results match quantitatively with the agent-based model. We thus provide an answer to the question of how to determine the population level drift velocity from the molecular mechanisms of chemotaxis for all range of chemical gradient, at least in the exponential environment. We focus on the exponential environment in the present paper. One interesting question is: what is the general type of chemical signalling environment where the parabolic or hyperbolic scaling can be valid. To address this question more tests have to be done. This is left for our future investigation.

\begin{figure} 
 \centering{
 \subfigure[]{\includegraphics[width=7cm]{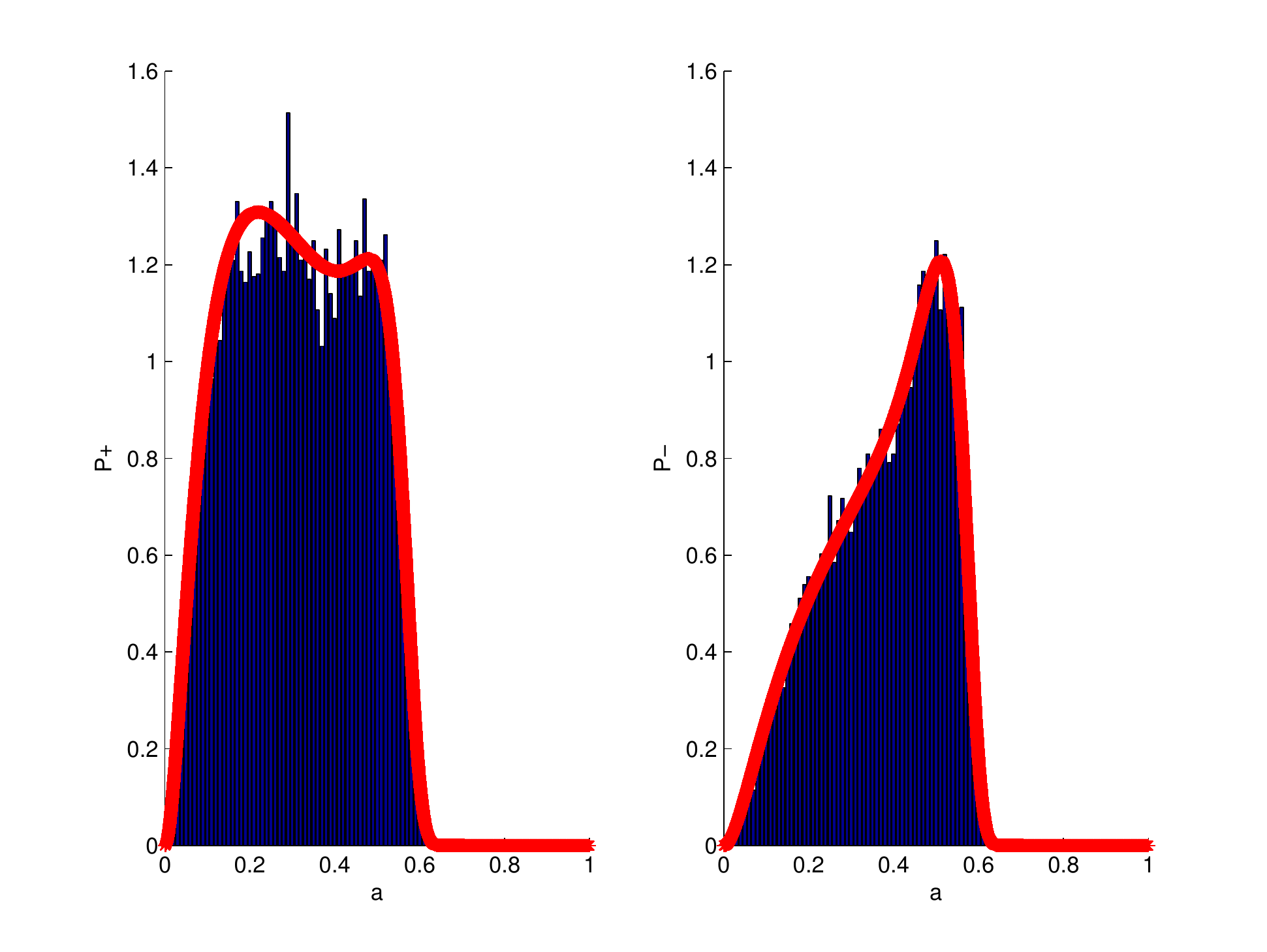}}
 \subfigure[] {\includegraphics[width=7cm]{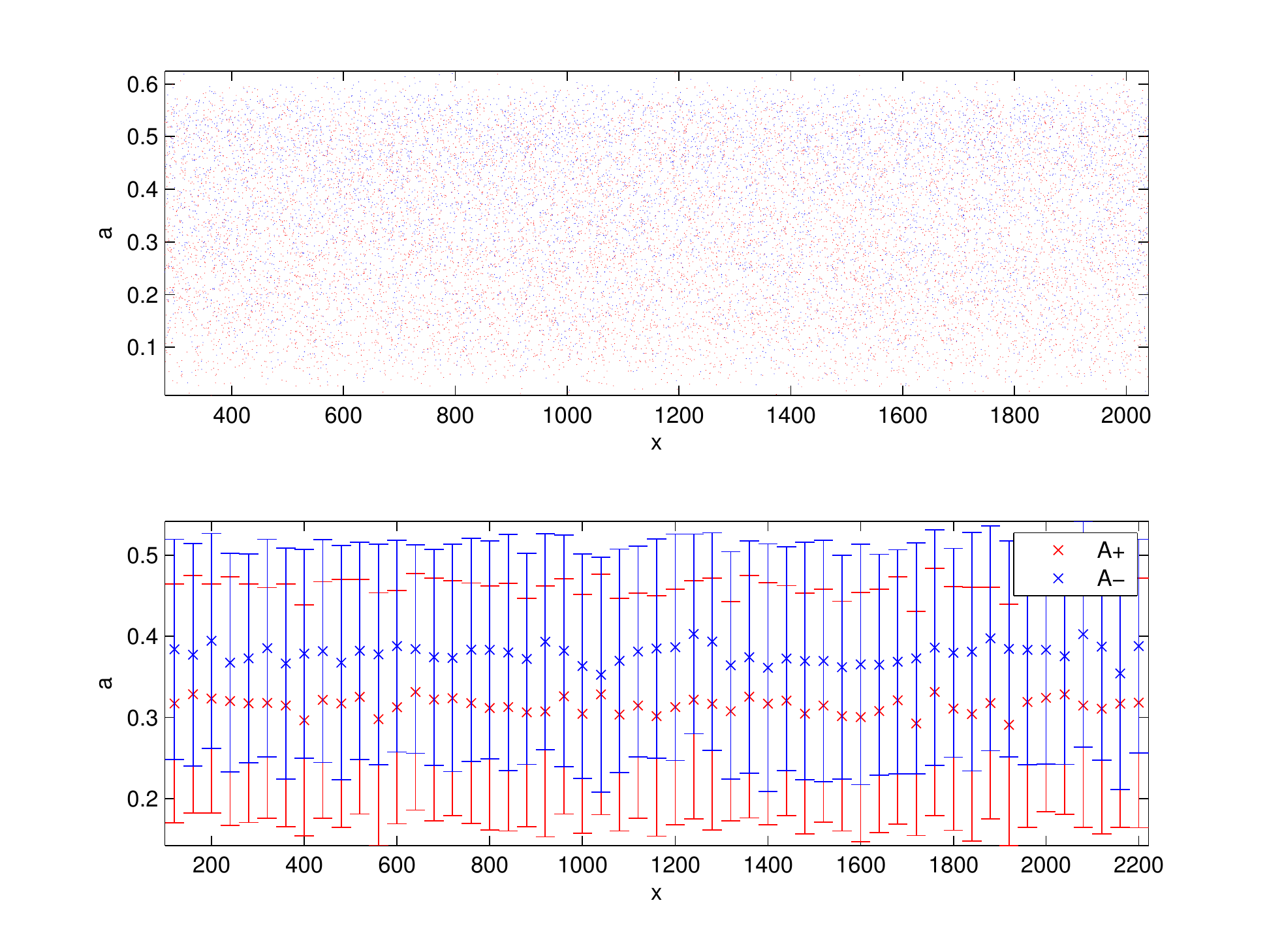}}
 \subfigure[] {\includegraphics[width=7cm]{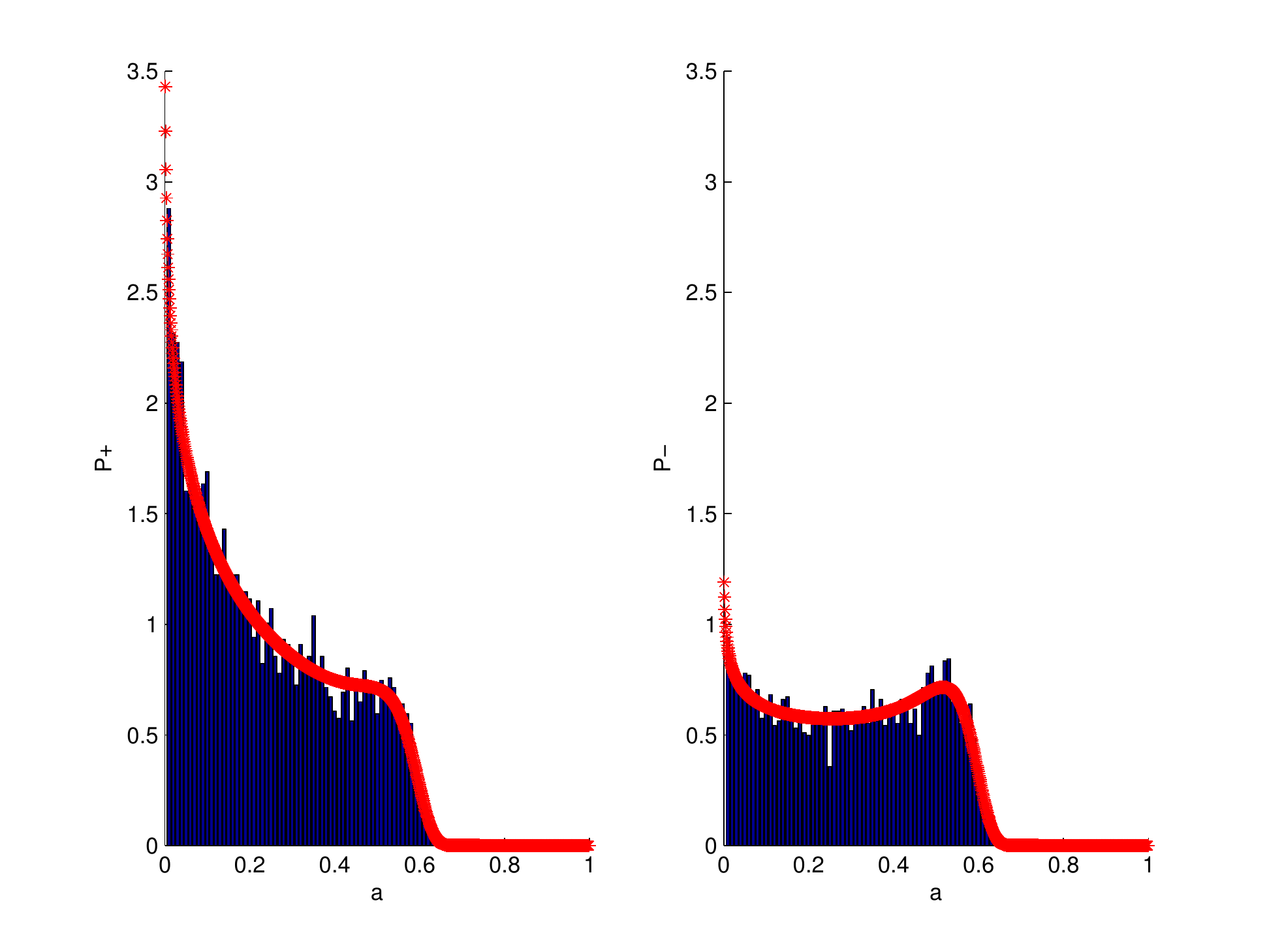}}
 \subfigure[] {\includegraphics[width=7cm]{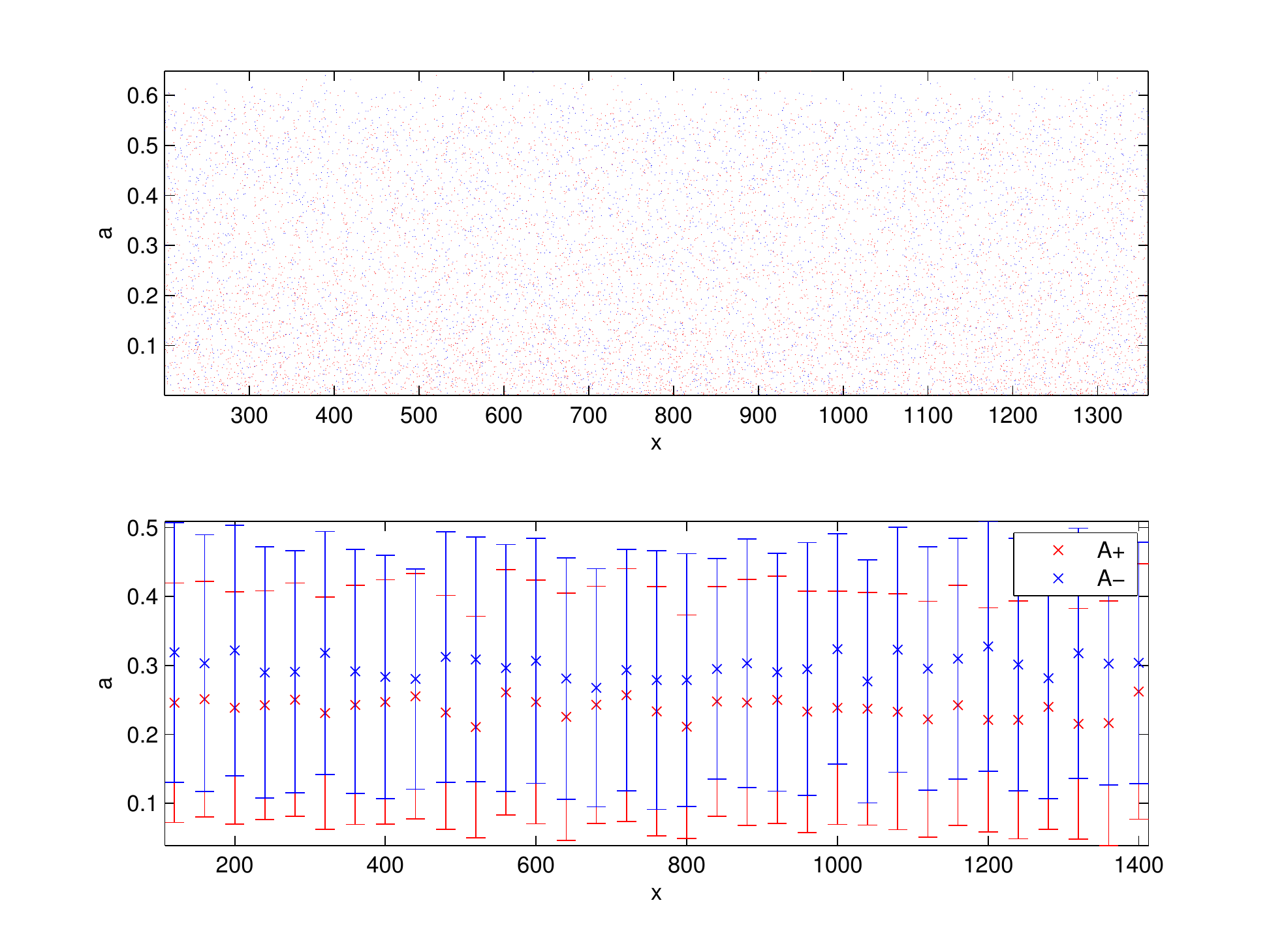}}
  \subfigure[] {  \includegraphics[width=7cm]{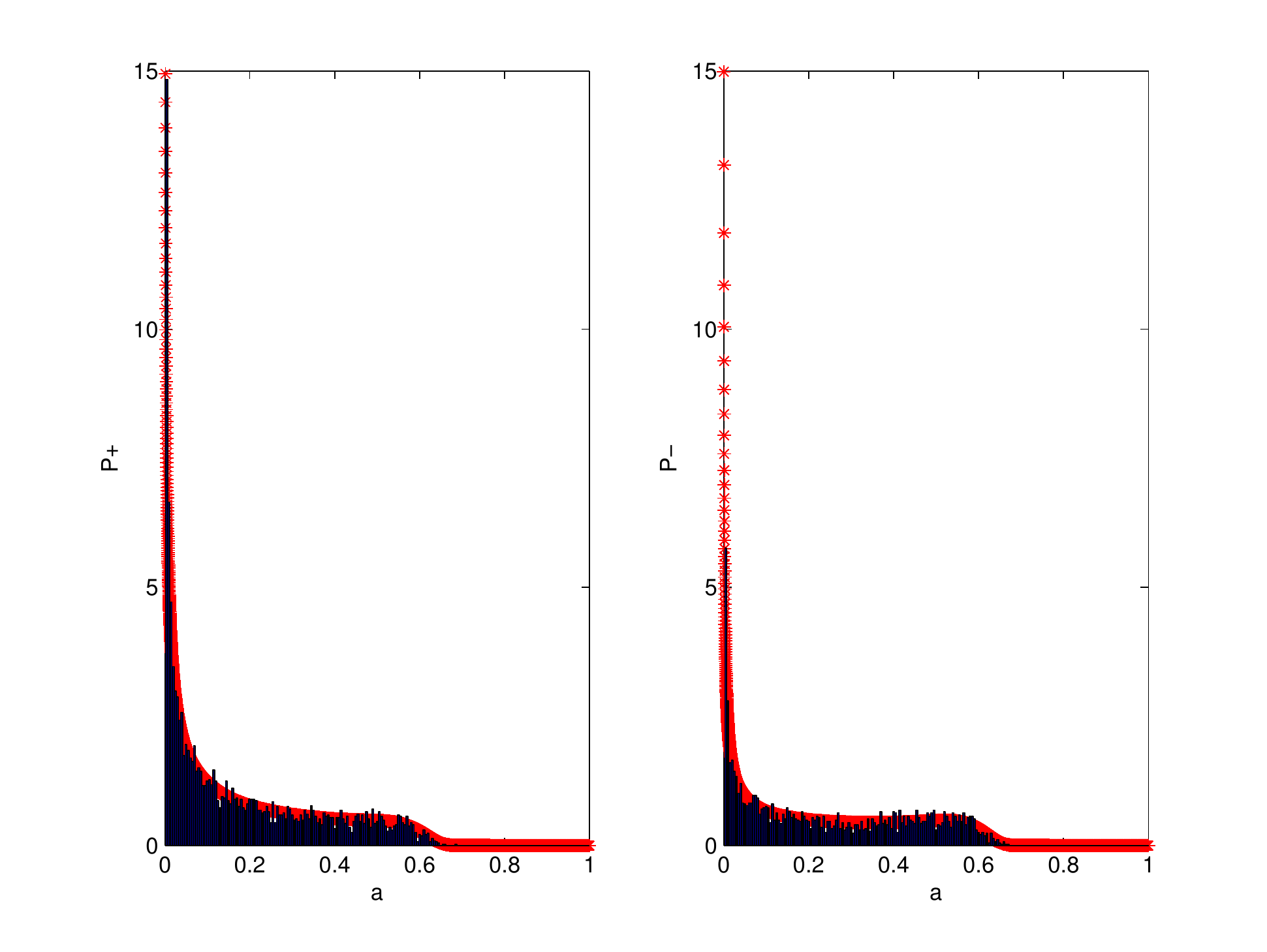}}
 \subfigure[]  {\includegraphics[width=7cm]{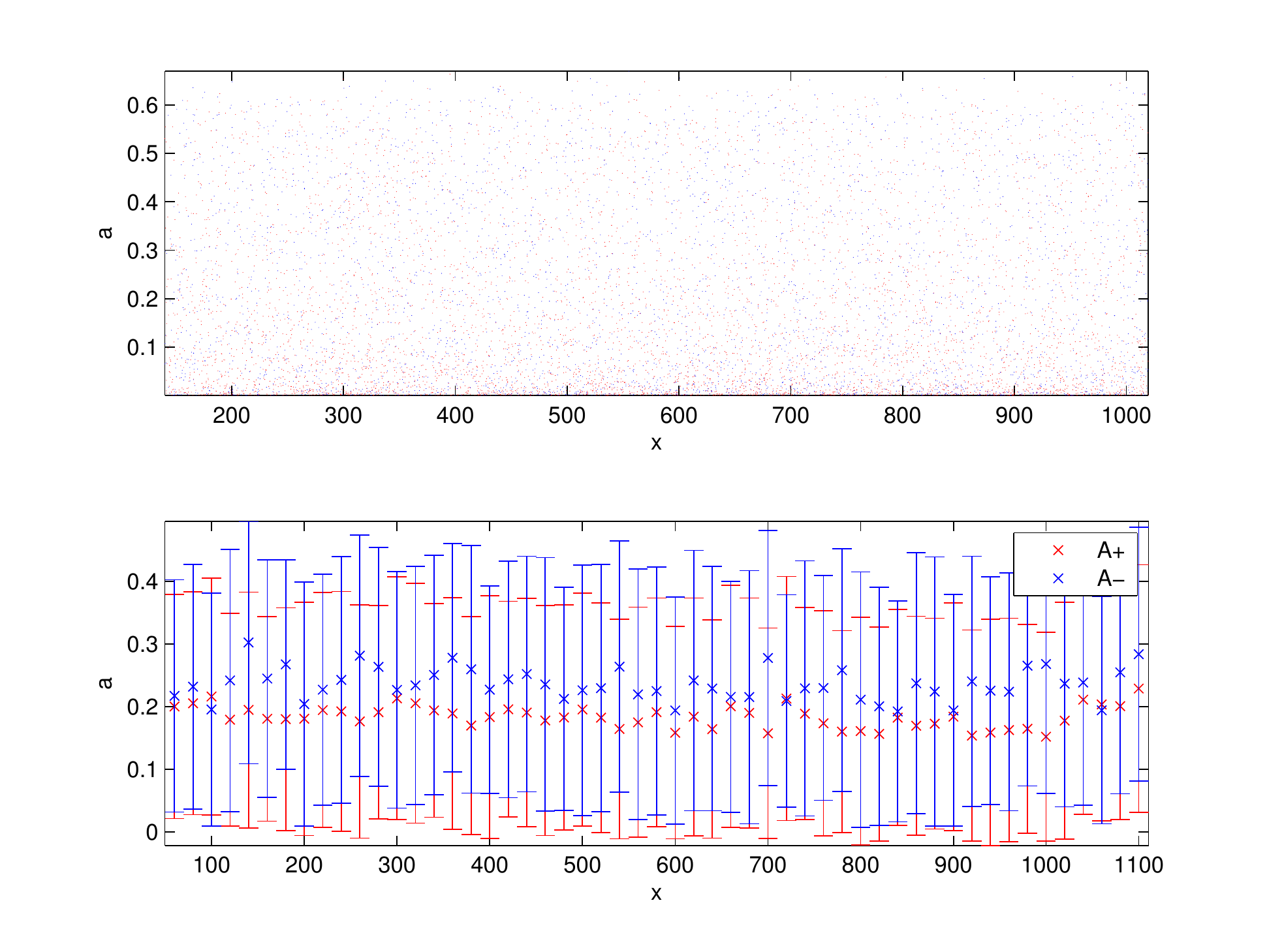}}
  } 
  \caption{The distribution of $\frac{q_0^+}{N \alpha_0 a (1-a)}$ (left) and $\frac{q_0^-}{N \alpha_0 a (1-a)}$ (right) in $a$ for different $G$. a) b): $G=1*10^{-3}$; c) d) : $G=1.5*10^{-3}$; e) f): $G=2*10^{-3}$ when $k_R = 0.005\mu m^{-1}$. These values all correspond to Case I with $g > 1$.
  a) c) e): The comparison of the distribution in $a$ for forward moving bacteria (left) and backward moving bacteria (right). The bars are the results from SPECS and the solid lines are from the analytical formula in \eqref{soln:q-0-plus-1}-\eqref{soln:q-0-minus-1}. b) d) f): The distribution of $a$ obtained by SPECS: red is for the forward moving bacteria and blue for the backward. The top subplots display the space distribution of bacteria with $a$ being the vertical axis. The bottom subplots give the mean and variance of $a$ at different positions. } \label{fig:case12}
\end{figure}

\begin{figure} 
 \centering{
 \subfigure[]{\includegraphics[width=7cm]{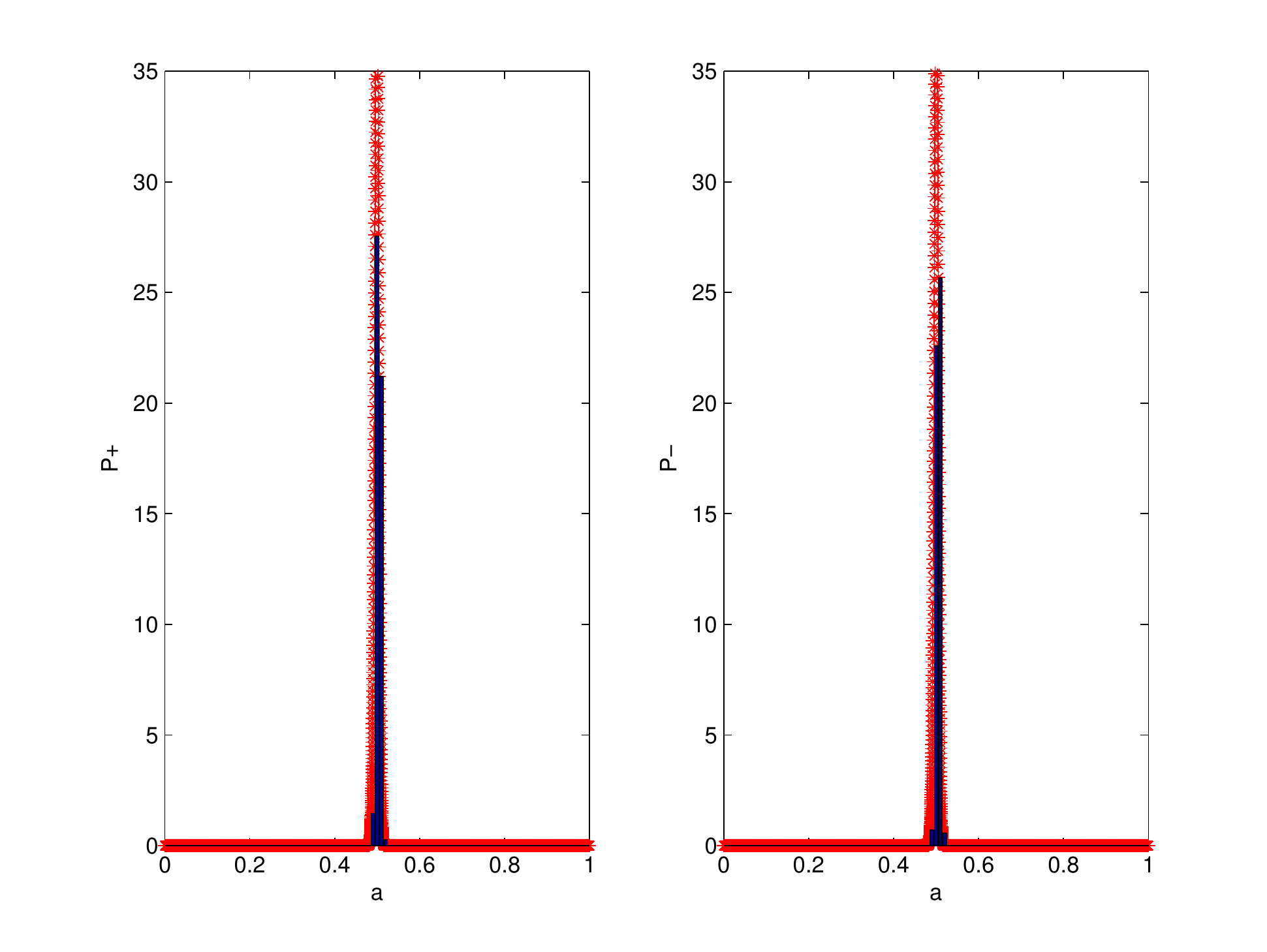}}
  \subfigure[]{ \includegraphics[width=7cm]{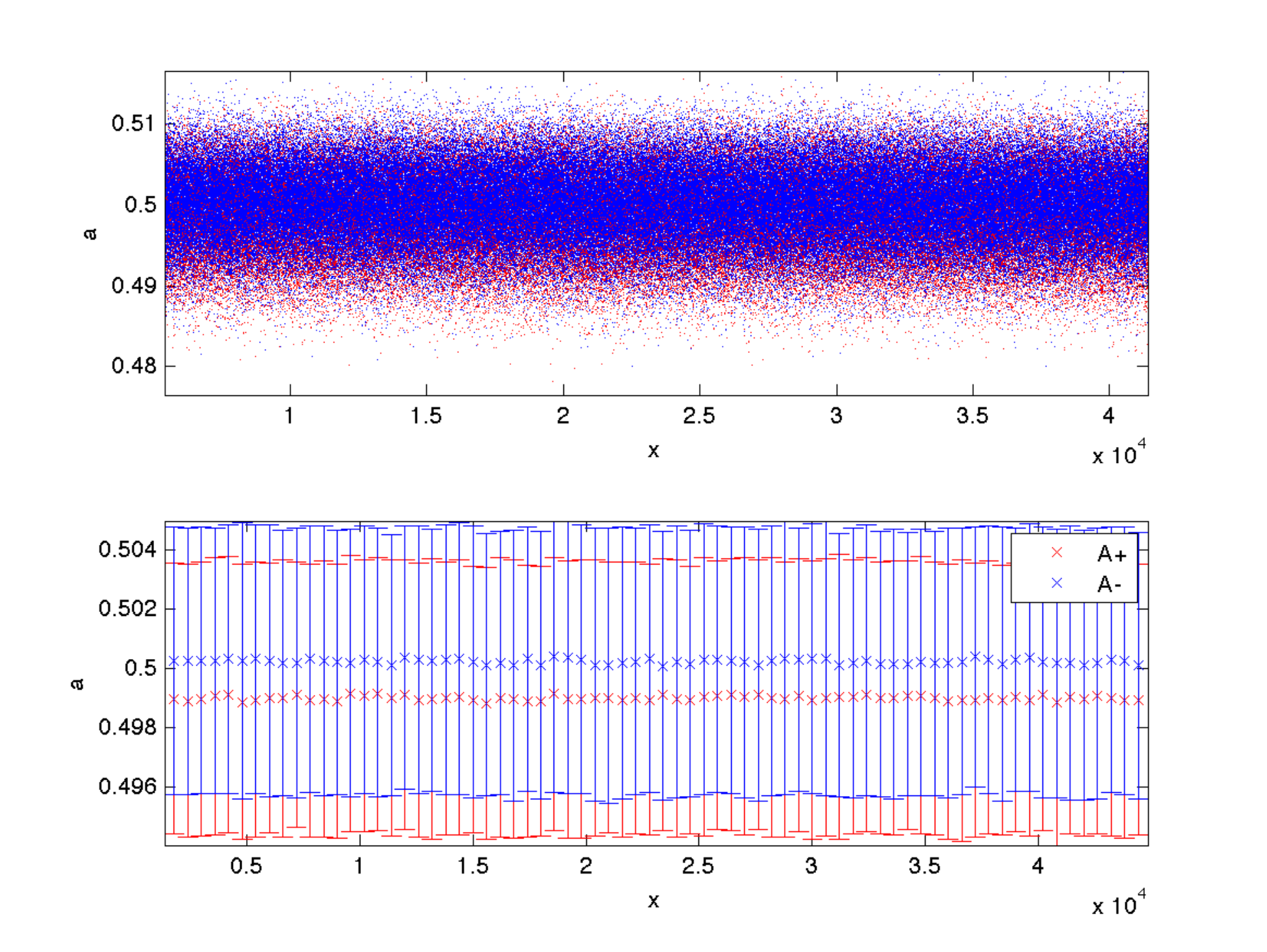}}
  \subfigure[]{ \includegraphics[width=7cm]{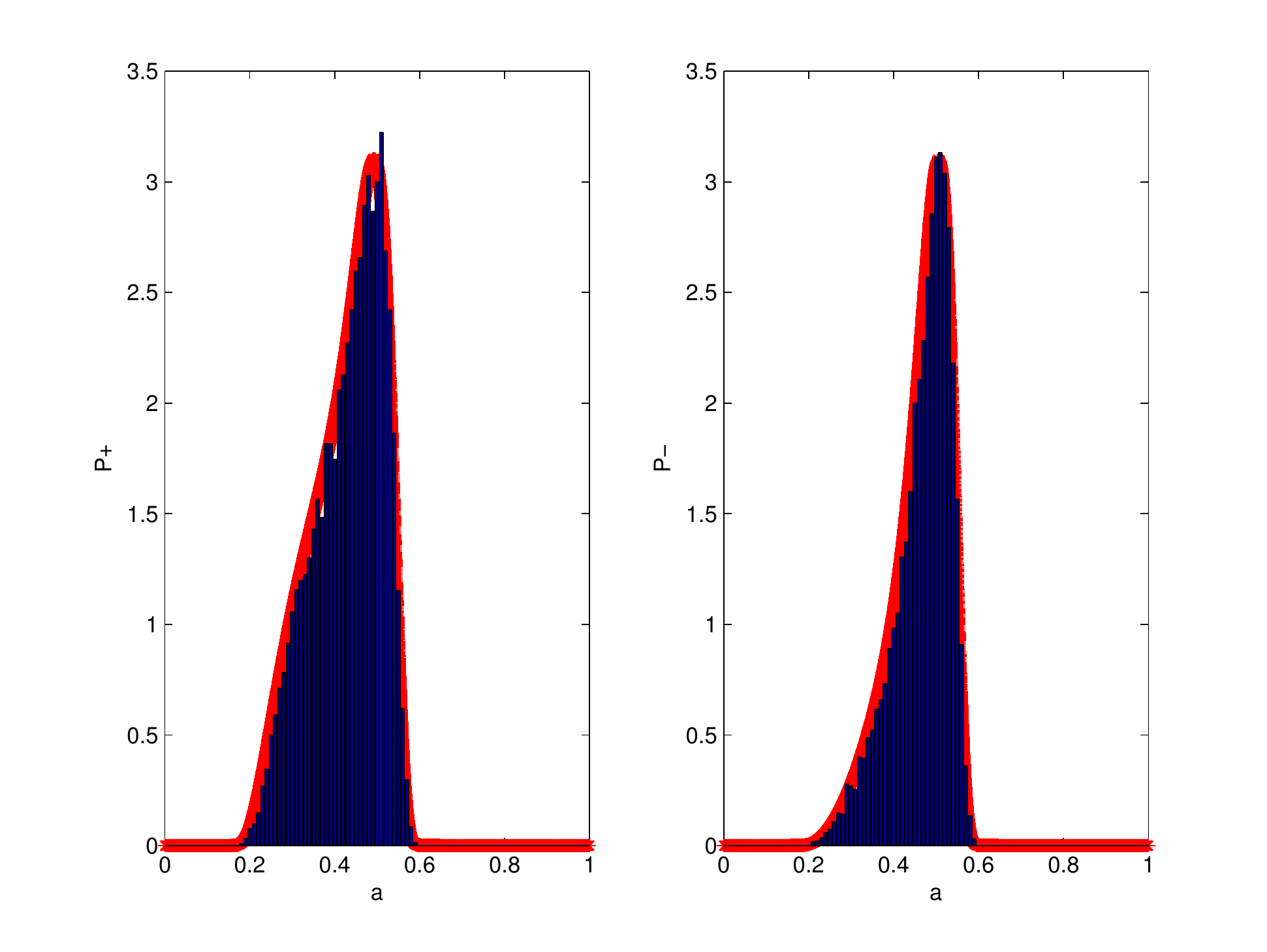}}
   \subfigure[]{ \includegraphics[width=7cm]{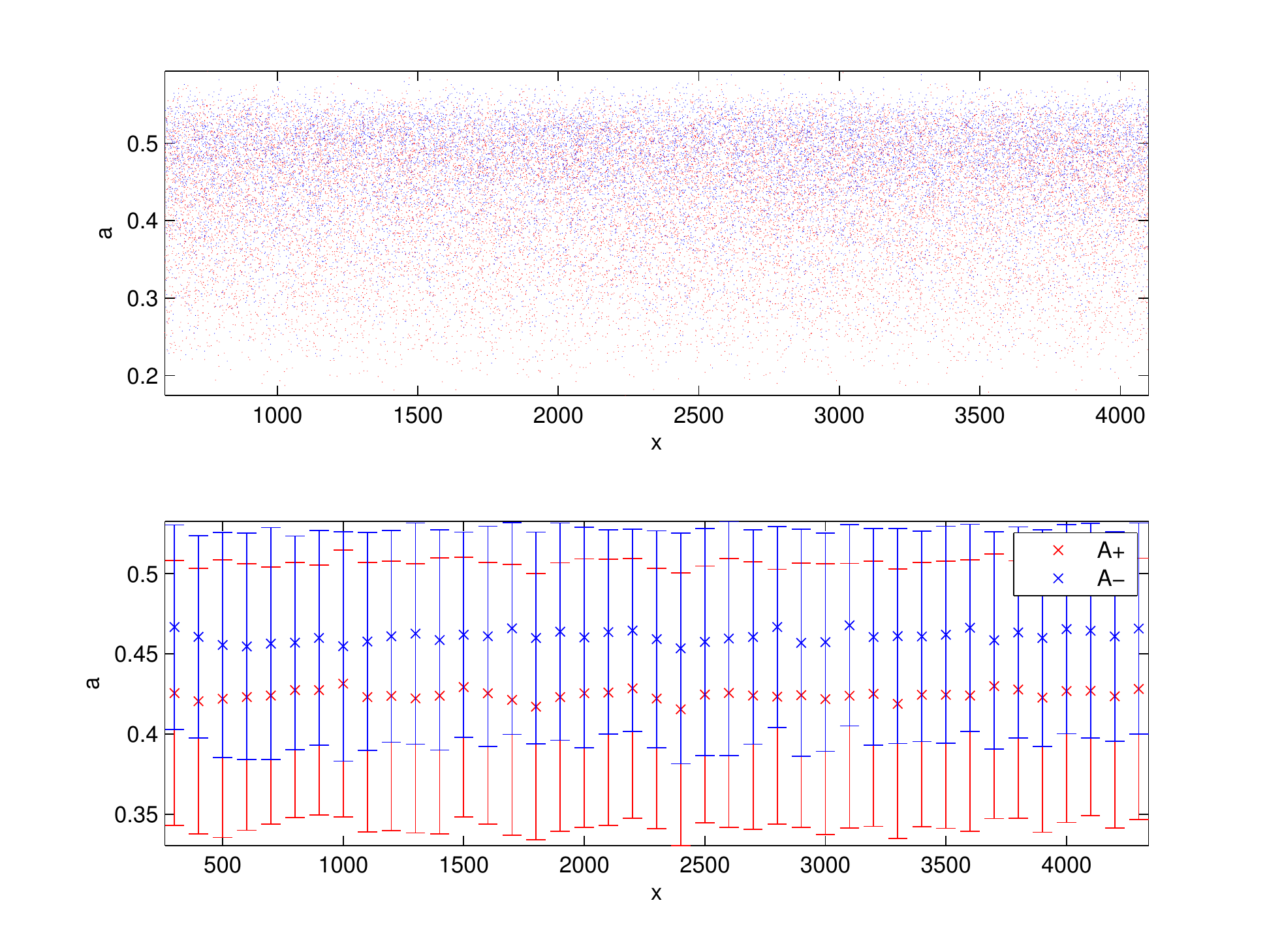}}
   }
\caption{The distribution of $\frac{q_0^+}{N \alpha_0 a (1-a)}$(left) and $\frac{q_0^-}{N \alpha_0 a (1-a)}$ (right) in $a$ for different $G$'s when  $k_R = 0.005\mu m^{-1}$. a) b): $G=5*10^{-5}$ (Case II with $0 < \mu <1$); c) d) : $G=5*10^{-4}$ (Case I with $0 < g < 1$).
  a) c): The comparison of the distribution in $a$ for forward moving bacteria (left) and backward moving bacteria (right). The bars are the results from SPECS and the solid lines are from the analytical formula in \eqref{soln:q-0-plus-1}-\eqref{soln:q-0-minus-1}. b) d): The distribution of $a$ obtained by SPECS: red is for the forward moving bacteria and blue for the backward. The top subplots display the space distribution of bacteria with $a$ being the vertical axis. The bottom subplots give the mean and variance of $a$ at different positions. } \label{fig:case34}
\end{figure}

\begin{figure}
 \centering{
  \subfigure[]{\includegraphics[width=5cm]{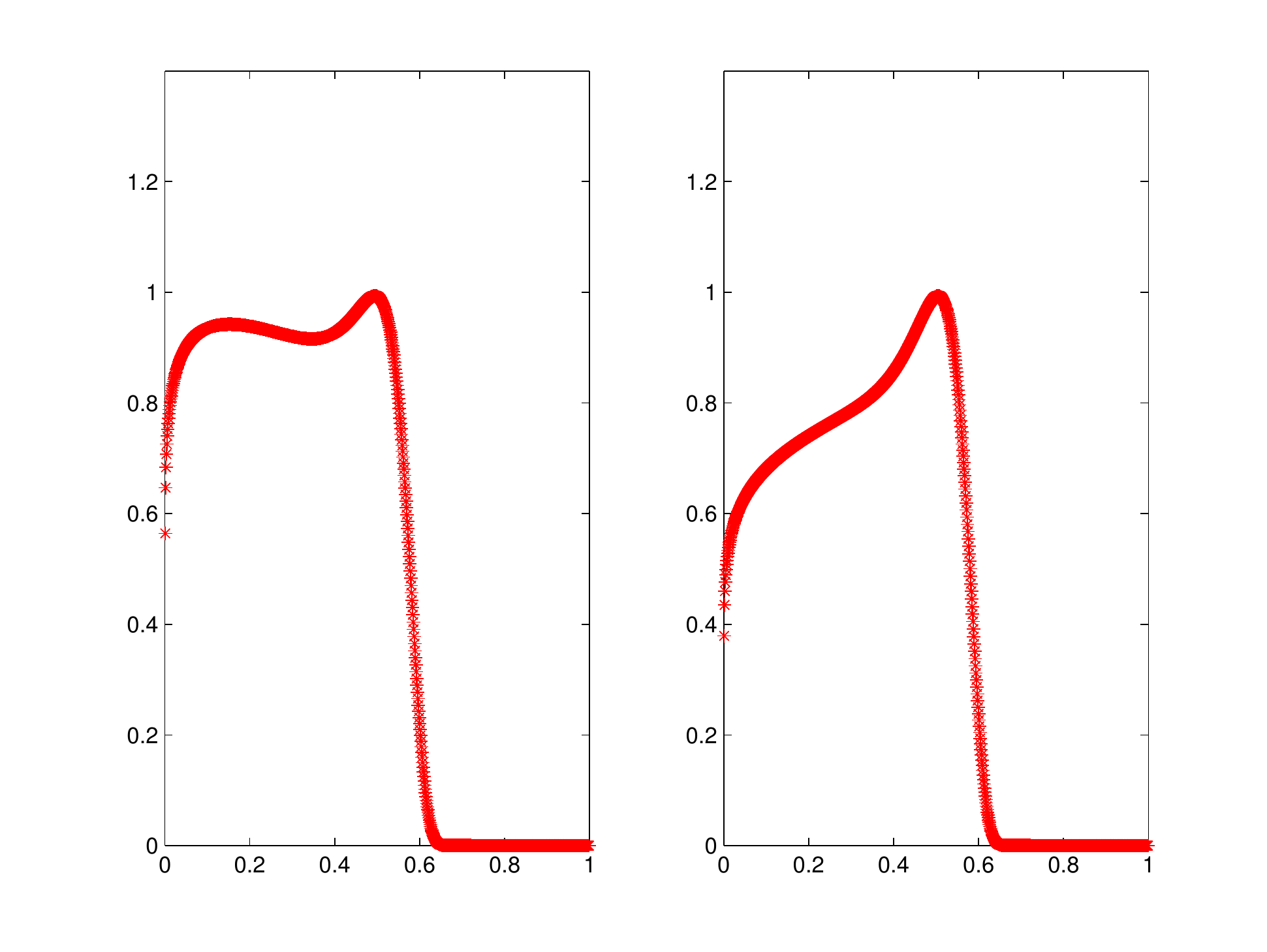}}
  \subfigure[]{ \includegraphics[width=5cm]{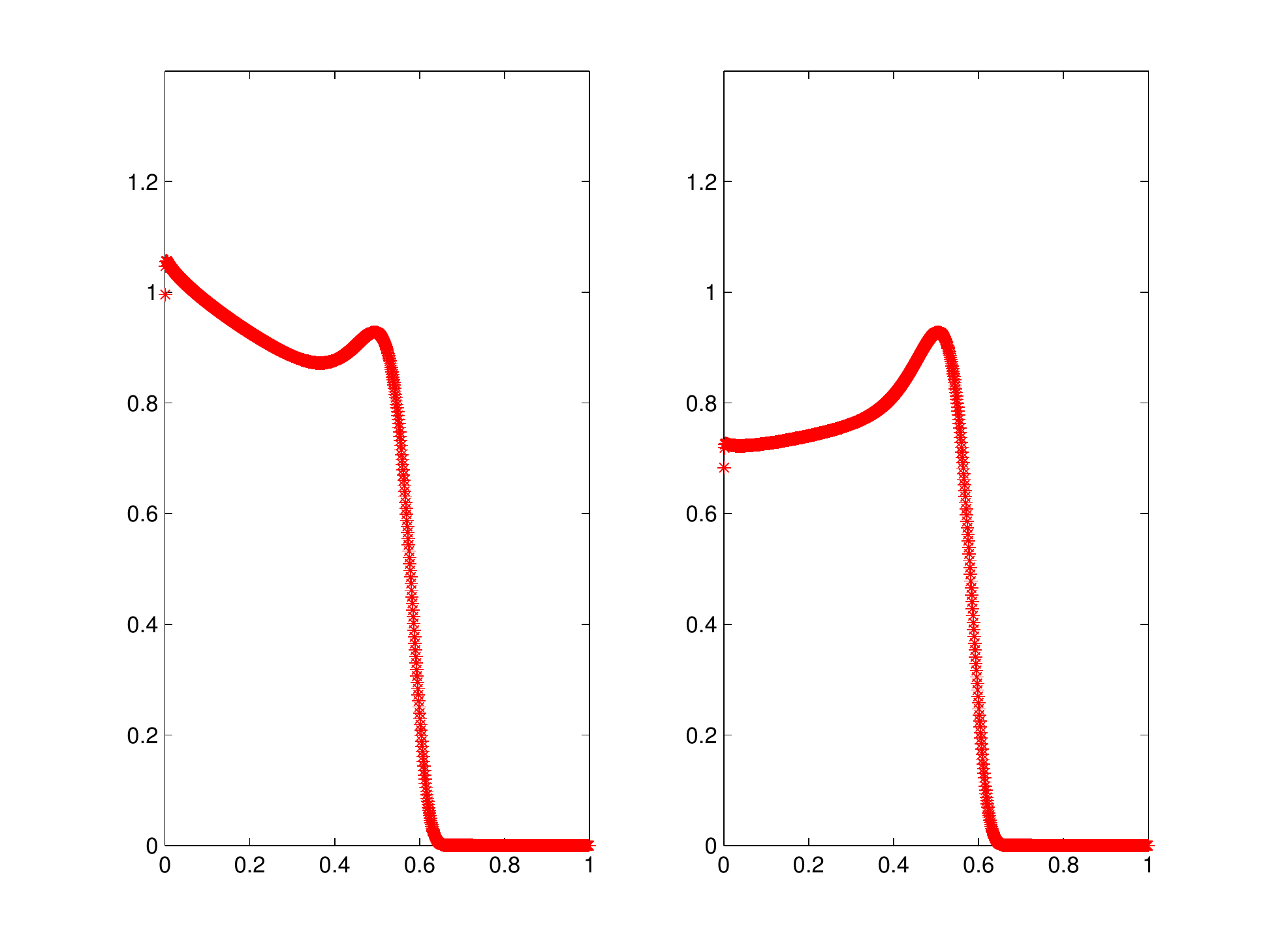}}
   \subfigure[]{\includegraphics[width=5cm]{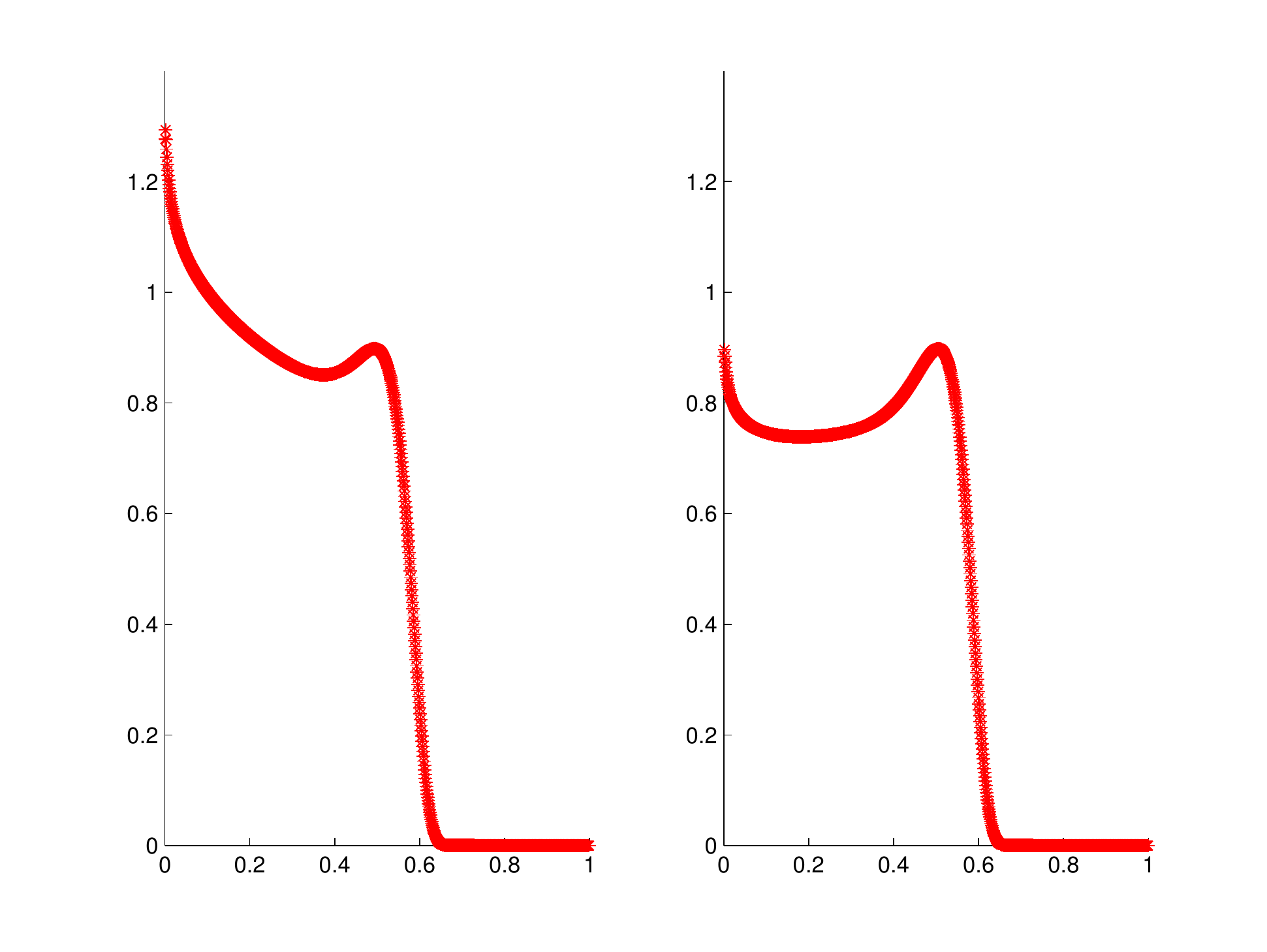}}
  } \caption{Each figure gives the distribution in $a$ for forward moving bacteria (left) and backward moving bacteria (right) when $k_R=0.0005s^{-1}$. Left: $G=3.7*10^{-4}$, $\theta_0=1.1072$; Middle: $G=3.9*10^{-4}$, $\theta_0=0.9925$; Right: $G=4.0*10^{-4}$, $\theta_0=0.9419$. All these values correspond to Case I with $g>1$.}\label{fig:disa0}
\end{figure}

\begin{figure}
 \centering{
  \subfigure[]{\includegraphics[width=7cm]{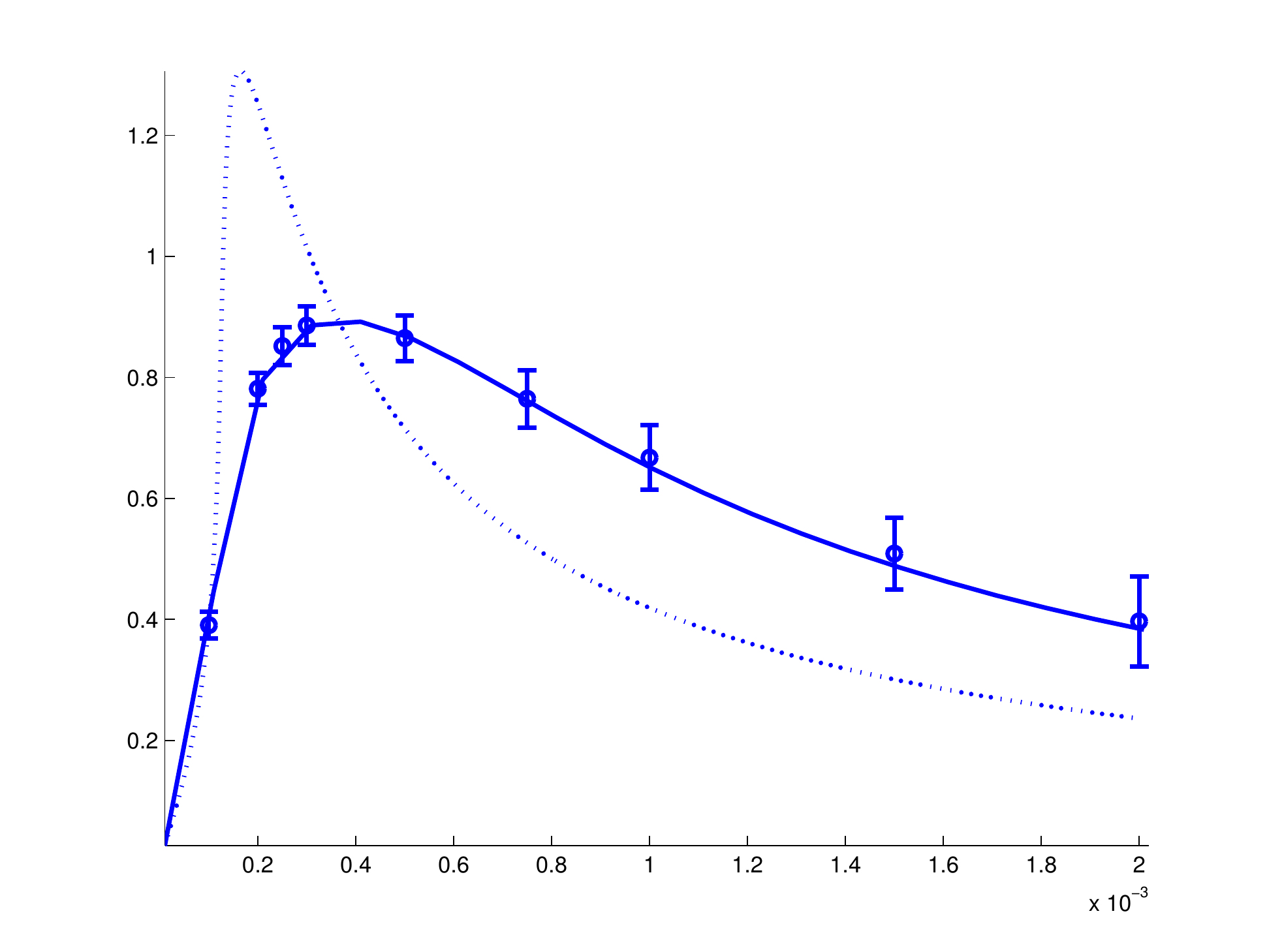}}
  \subfigure[]{ \includegraphics[width=7cm]{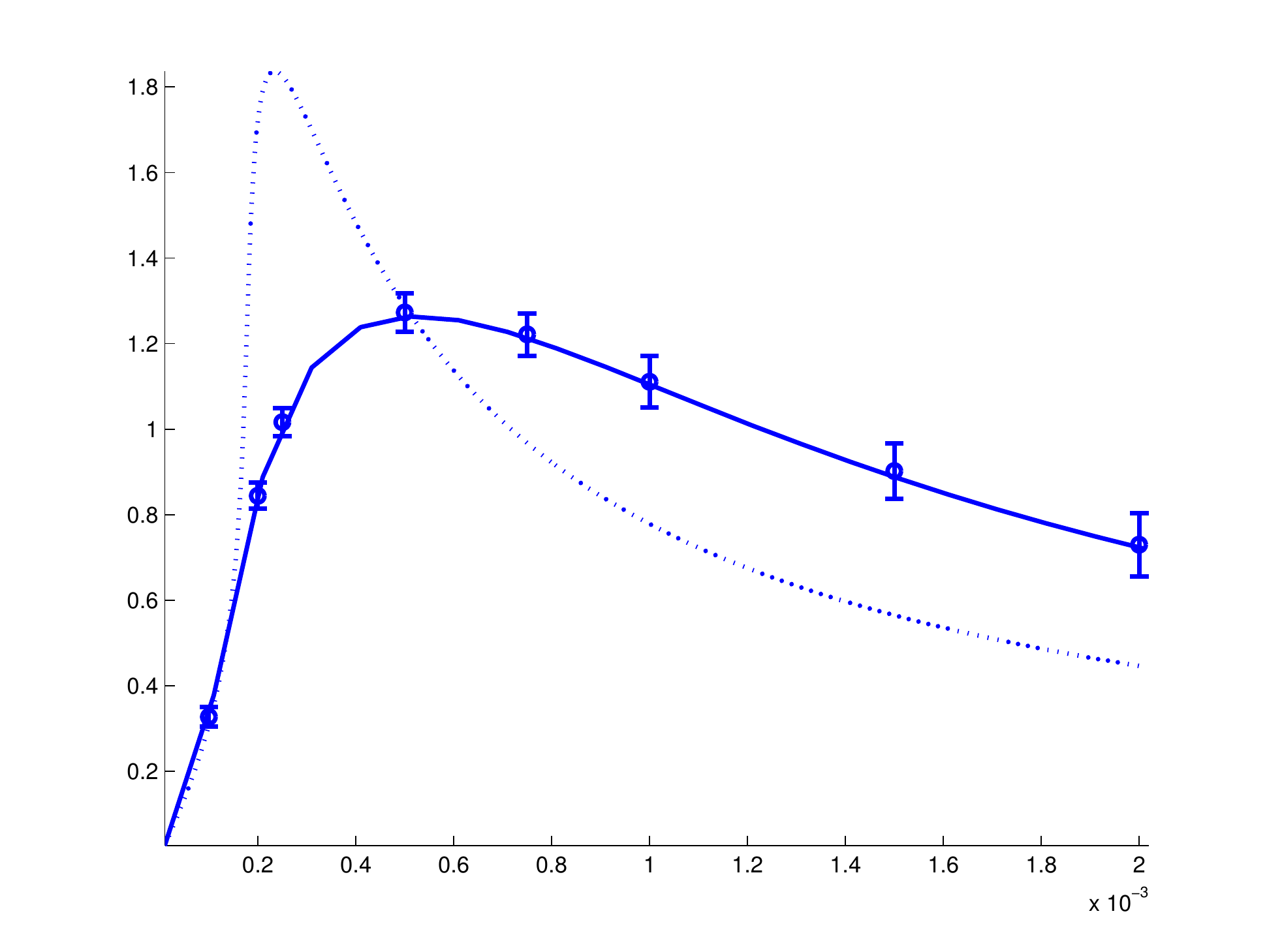}}
  \subfigure[]{\includegraphics[width=7cm]{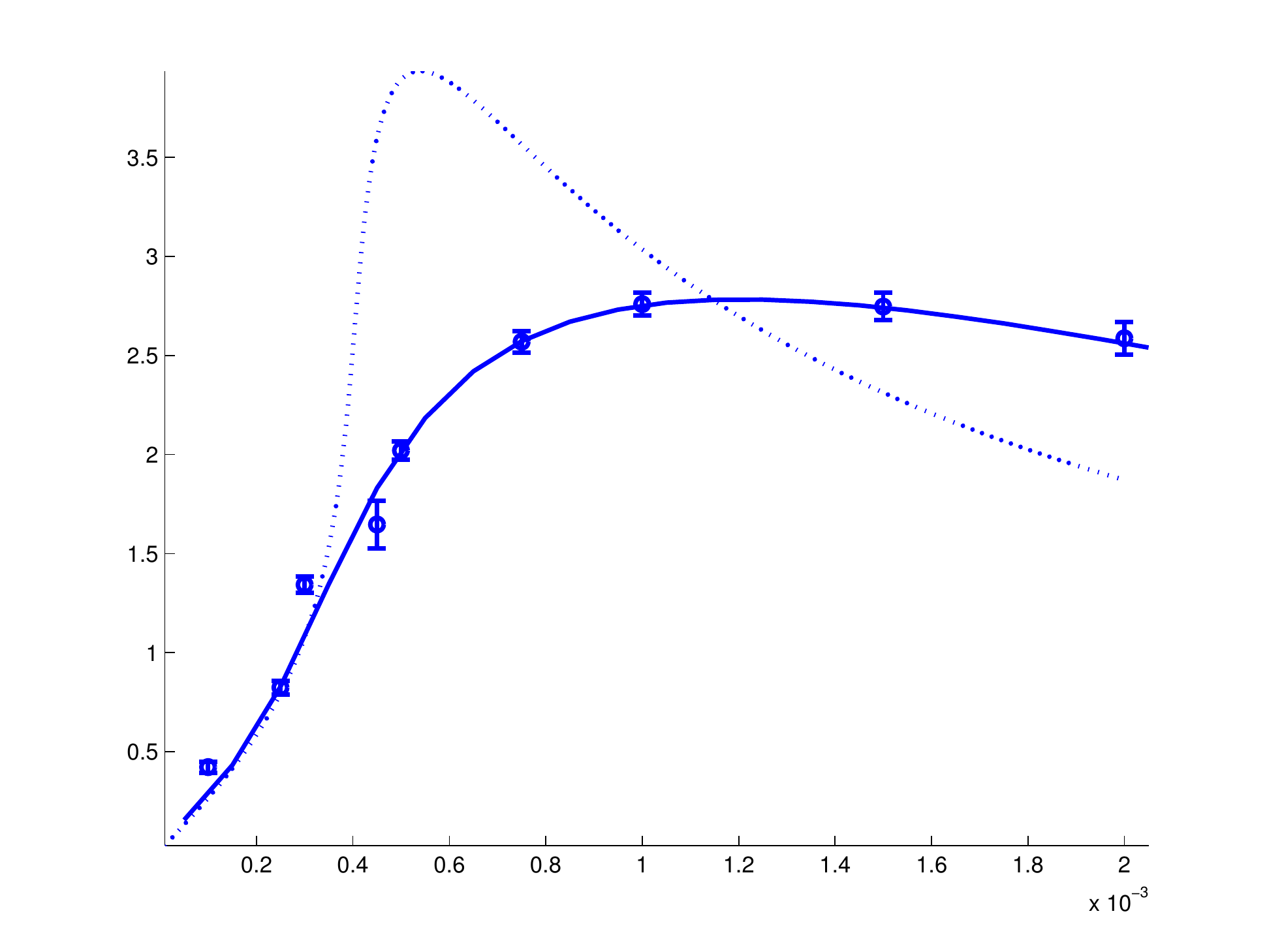}}
  \subfigure[]{\includegraphics[width=7cm]{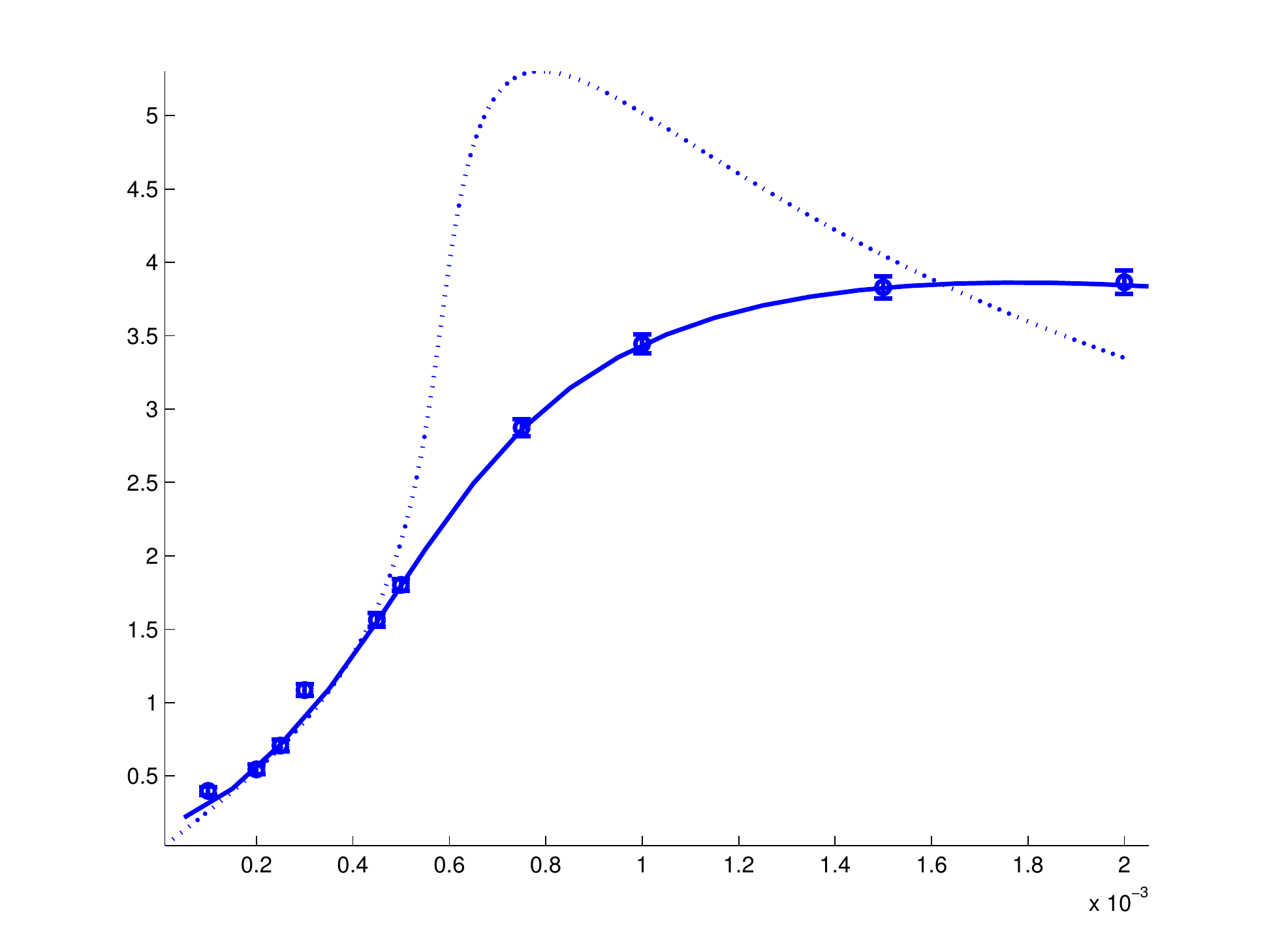}}
  } \caption{ Average chemotaxis velocity in exponential concentration gradient for different $k_R$. a): $k_R=0.0005s^{-1}$; b): $k_R=0.001s^{-1}$; c): $k_R=0.005s^{-1}$; d): $k_R=0.01s^{-1}$. Here the solid lines are calculated from the analytical formula \eqref{eq:kappa1}, the error bars with circles are the results by SPECS simulations and the dotted lines are the prediction by PBMFT.}\label{fig:vd}
\end{figure}



\end{document}